\documentclass[12pt]{article}
\usepackage{amsmath, amssymb, amsfonts, amsthm}
\usepackage{graphicx}
\usepackage{color}
\usepackage[all]{xy}
\newdir{d}{{\cong}}
\usepackage[margin=1in]{geometry}
\usepackage{ stmaryrd }
\usepackage{ textcomp }
\usepackage{enumerate}
\usepackage{mathrsfs}
\usepackage{tikz}
\theoremstyle{plain}
\newtheorem{thm}{Theorem}[section]
\newtheorem{lemma}[thm]{Lemma}
\newtheorem{corollary}[thm]{Corollary}
\newtheorem{proposition}[thm]{Proposition}
\theoremstyle{definition}
\newtheorem{definition}[thm]{Definition}
\newtheorem{remark}[thm]{Remark}
\theoremstyle{example}

\theoremstyle{plain}

\title{$E$-theory Spectra for graded $C^\ast$-algebras}
\author{Sarah L. Browne}

\begin{document}
\maketitle

\begin{abstract}
This paper brings together $C^\ast$-algebras and algebraic topology in terms of viewing a $C^\ast$-algebraic invariant in terms of a topological spectrum. $E$-theory, $E(A,B)$, is a bivariant functor in the sense that is a cohomology functor in the first variable and a homology functor in the second variable but underlying goes from the category of separable $C^\ast$-algebras and $\ast$-homomorphisms to the category of abelian groups and group homomorphisms. 
Here we create a generalisation of a orthogonal spectrum to quasi-topological spaces for $E$-theory . This includes a rich product structure in the context of graded separable $C^\ast$-algebras. 
\end{abstract} 

\section*{Introduction}
Separable $C^\ast$-algebras are analytical objects used in non-commutative geometry. $E$-theory, $E(A,B)$, is a invariant of $C^\ast$-algebras and is a bivariant functor from the category of separable $C^\ast$-algebras and $\ast$-homomorphisms to the category of abelian groups and group homomorphisms. Here we will only consider real graded $C^\ast$-algebras since then we have a useful tool for application in differential geometry. $E$-theory is a cohomology theory in its first variable and a homology theory in its second variable. It was first defined by Higson~\cite{Hig90} in a categorical manner and Connes and Higson~\cite{CH90} gave a concrete description using homotopy classes of certain morphisms of $C^\ast$-algebras. This will form section~1. 

Spectrum is a topological notion used to represent stable homotopy theories. Here we require the notion of a quasi-spectrum, which is considered as a spectrum over quasi-topological spaces. We incorporate an action of the orthogonal group and create an orthogonal quasi-spectrum. This has a rich product structure as we will see. 

The main construction encodes the Bott periodicity of $E$-theory proving in the complex graded case by Guentner-Higson~\cite{GH04} and in the real case by Browne~\cite{Bro}, as an $\Omega$-quasi-spectrum coming from weak equivalences. Additionally we also encode the product structure of $E$-theory. This then will give us a stable homotopy theory. It is worth noting that the spectrum will be dependent on $C^\ast$-algebras just like the K-theory spectrum for $C^\ast$-categories is dependent on $C^\ast$-categories defined by Paul D.~Mitchener~\cite{Mit01}. The author likes to think of this as a ``local'' spectrum and it is suited for possible applications to positive scalar curvature by generalising work of Weiss-Williams~\cite{WW95}

Further we include a way of connecting $K$-theory spectrum and $K$-homology spectrum using the relations of $E$-theory with both $K$-theory and $K$-homology in the final section of the paper.

\section*{Acknowledgements}
The author would like to thank their PhD supervisor, Paul Mitchener for all the valuable meetings and discussions during her PhD. The author would also like to thank the EPSRC for funding her PhD which made this possible. Lastly the author wishes to thank many mathematicians for conversations and particularly Jamie Gabe for his insight during a visit at the University of Southampton.  

\section{$E$-theory}
This section details the analytic necessities which we need to inlude in our definition of the quasi-spectrum for $E$-theory. 
A \emph{complex $C^\ast$-algebra} is a complex Banach space with involution $\ast$ satisfying for all $a \in A$, the $C^\ast$-identity $||a^\ast a || = ||a||^2$. A \emph{real $C^\ast$-algebra} is a real Banach space with an involution, satisfying th $C^\ast$-identity and additonally that the element $1+a^\ast a$ is invertible in $A$ for all $a \in A$. A $C^\ast$-algebra is called \emph{separable} if its underlying topological space has a countable dense subset and all of our $C^\ast$-algebras will be separable.
A map $f \colon A \rightarrow B$ of real $C^\ast$-algebras is called a \emph{$\ast$-homomorphism} and is an algebra homomorphism and satisfies $f(a^\ast) = (f(a))^\ast$ for all $a \in A$. 
A \emph{grading} on a $C^\ast$-algebra $A$ is an automorphism $\delta_A \colon A \rightarrow A$ such that $\delta_{A}^2 = 1$. We can also think of a grading by allocating an even and odd notion on $A$. That is we have $A = A_{\text{even}} \oplus A_{\text{odd}}$, where
\[A_{\text{even}} = \{a \in A ~ | ~ \delta_{A} (a) = a \} ~ \text{and} ~  A_{\text{odd}} = \{a \in A ~ | ~ \delta_{A} (a) = -a \}.\]
Then additionally we define the \emph{degree} of an element $a \in A$ by: 
\begin{equation*}
\text{deg}(a) =
\begin{cases}
0,& \text{if}\; a \in A_{\text{even}} \\
1,& \text{if} \; a \in A_{\text{odd}}.
\end{cases}
\end{equation*}
An important example of a graded real $C^\ast$-algebra is the algebra of continuous real-valued functions vanishing at infinity, $\mathcal{S} = C_0(\mathbb{R})$, under the supremum norm and grading given by $\delta(f)(x) = f(-x)$ for all $x \in \mathbb{R}$. Denote this algebra by $\mathcal{S}$. Also we can consider the \emph{suspension} of a $C^\ast$-algebra
\[\Sigma A = \{f \colon [0,1] \rightarrow A ~|~ f(0)=f(1)=0\}.\]
If $A$ is graded then $\Sigma A$ has grading coming from $A$. Also we can consider the $n$-fold suspension, $\Sigma^n A = \Sigma^{n-1} \Sigma A$. 

Now we should note that a \emph{graded $\ast$-homomorphism} is a $\ast$-homomorphism that preserves the grading. That is, if an element is even then its image is even and if an element is odd then its image is odd. Furthermore we can define a different type of morphism between $C^\ast$-algebras:- 

\begin{definition}
Let $A$ and $B$ be real graded $C^\ast$-algebras with gradings $\delta_A$ and $\delta_B$ respectively. A \emph{graded asymptotic morphism} $\varphi \colon A \dashrightarrow B$ is a family of functions $\{\varphi_t\}_{t \in [1,\infty)}$ such that: 
\begin{enumerate}
\item the map $t \mapsto \varphi_t(a)$, from $[1,\infty)$ to  $B$ is continuous and bounded for each $a \in A$, 
\item $\lim_{t \to \infty} || \varphi_t (ab) - \varphi_t (a)\varphi_t(b) || = 0,$ for each $a,b \in A$, 
\item $\lim_{t \to \infty} || \varphi_t (a + \lambda b) - \varphi_t (a) - \lambda \varphi_t (b) || = 0,$ for each $a,b \in A$, $\lambda \in \mathbb{R}$, 
\item $\lim_{t \to \infty} || \varphi_t (a^*) - \varphi_t (a)^* || = 0,$ for each $a \in A$, 
\item $\lim_{t \to \infty} || \varphi_t(\delta_A(a)) - \delta_B(\varphi_t(a)) || = 0$ for each $a \in A$. 
\end{enumerate}
\end{definition}
Denote the set of these by $\text{Asy}_g(A,B)$. 
It will be useful to consider the \emph{graded tensor product}: 
\begin{definition}
Let $A$ and $B$ be graded $C^\ast$-algebras with gradings $\delta_A$ and $\delta_B$ respectively. Then  define $A \widehat{\otimes} B$ to be the completion of the algebraic tensor product of $A$ and $B$ in the norm 
\[|| \sum_i a_i \otimes b_i || = \sup_{\varphi, \psi} || \sum_i \varphi(a_i) \psi(b_i)||\]  
where $\varphi \colon A \dashrightarrow C$, $\psi \colon A \dashrightarrow C$ are graded asymptotic morphisms to a common graded $C^\ast$-algebra $C$. 
We equip $A \widehat{\otimes} B$ with involution, multiplication and grading defined by:
\begin{enumerate}
\item $(a \widehat{\otimes} b)^{\ast} = (-1)^{\text{deg}(a) \text{deg}(b)} a^{\ast} \otimes~  b^{\ast}$
\item $(a \widehat{\otimes} b)(c \widehat{\otimes} d) =  (-1)^{\text{deg}(b) \text{deg}(c)}(ac \otimes bd)$
\item $\gamma(a \widehat{\otimes} b) = \alpha(a) \otimes \beta(b)$
\end{enumerate}
Extending by linearity gives $A \widehat{\otimes} B$. 
\end{definition}
We have the following result, Lemma~4.5 in~\cite{GHT00}, for asymptotic morphisms:
\begin{lemma} \label{tensorequivalence}
Let $\varphi \colon A_1 \dashrightarrow A_2$ and $\psi \colon B_1 \dashrightarrow B_2$ be (graded) asymptotic morphisms, then the compositions 
\[A_1 \widehat{\otimes} B_1 \xrightarrow{\varphi \widehat{\otimes} 1}A_2 \widehat{\otimes} B_1 \xrightarrow{1 \widehat{\otimes} \psi} A_2 \widehat{\otimes} B_2,\]
and
\[A_1 \widehat{\otimes} B_1 \xrightarrow{1 \widehat{\otimes} \psi}A_1 \widehat{\otimes} B_2 \xrightarrow{\varphi \widehat{\otimes} 1} A_2 \widehat{\otimes} B_2,\]
are equal. The $1$ symbolises the relevant identity morphism. 
\end{lemma}
\qed 

Now $E$-theory is built out of homotopy classes so we need the notion of a homotopy of graded asymptotic morphism:- 

\begin{definition}
A \emph{homotopy} of graded asymptotic morphisms $\varphi_t, \psi_t \colon A \dashrightarrow B$ is a graded asymptotic homomorphism $\theta_t \colon A \dashrightarrow C([0,1],B)$ such that 
\[\theta_t(a)(0) = \varphi_t(a) ~ \text{and} ~ \theta_t(a)(1) = \psi_t(a).\]
\end{definition}
Denote the set of homootopy classes of asymptotic morphisms from $A$ to $B$ by $\llbracket A,B \rrbracket$. 

We require the next remark to impose a particular grading on compact operators.
\begin{remark}\label{grading on compacts}
Let $\mathcal{H}$ be a Hilbert space equipped with the orthogonal decomposition 
\[\mathcal{H} = \mathcal{H}_0 \oplus \mathcal{H}_1,\]
where $\mathcal{H}_0$ denotes the even elements and $\mathcal{H}_1$ denotes the odd elements.
Then the $C^\ast$-algebra $\mathcal{K}(\mathcal{H})$ of compact operators on such a Hilbert space is graded. For this grading, we consider $2 \times 2$ matrices of operators where the diagonal matrices are even and the off diagonal ones are odd. 
That is we have a grading \[\beta \colon \mathcal{K}(\mathcal{H}) \rightarrow \mathcal{K}(\mathcal{H})\] defined by 
\begin{equation*}
\beta(T) =
\begin{cases}
T &  \; \text{if} \; T \; \text{is even} \\
-T &  \; \text{if} \; T \; \text{is odd}. 
\end{cases}
\end{equation*}
\end{remark}

Then the graded $E$-theory groups are defined by: 
\[E^n_g(A,B) = \llbracket \mathcal{S} \widehat{\otimes} A \widehat{\otimes} \mathcal{K}(\mathcal{H}), \Sigma^n B  \widehat{\otimes} \mathcal{K}(\mathcal{H}) \rrbracket, \]
where $\Sigma^n B$ for $n \geq 0$ is the $n$-fold suspension of a $C^\ast$-algebra and it's grading comes from $B$, and where $\mathcal{K}(\mathcal{H})$ denotes the compact operators on some real graded Hilbert space with grading as detailed in Remark~\ref{grading on compacts}. 
If we have a graded $\ast$-homomorphisms $\varphi, \psi \colon A \rightarrow B$, we can also form homotopy classes of these, by setting $\varphi = \varphi_t$, $\psi= \psi_t$ for all $t \in [1, \infty)$. Denote the set of homotopy classes of $\ast$-homomorphisms from $A$ to $B$ by $[A,B]$.

It will become important to consider the graded $\ast$-homomorphism
\[\Delta \colon \mathcal{S} \rightarrow \mathcal{S} \widehat{\otimes} \mathcal{S},\]
defined in~\cite{GH04} and also detailed in~\cite{Bro}. The idea is to restrict to the set of continuous functions on the interval $[-R,R]$, denoted by $S_R$ and use functional calculus and define 
\[f(X_R \widehat{\otimes} 1 + 1 \widehat{\otimes} X_R) \in S_R \widehat{\otimes} S_R,\]
and thereafter obtain a graded $\ast$-homomorphism.

Now in $E$-theory we have a Bott periodicity result that can be found in~\cite{Bro}. Let $\mathbb{F} = \mathbb{R}$ or $\mathbb{C}$. Before we state the result we note that $\mathbb{F}_{n,0}$ is the \emph{Clifford algebra} on n generators $e_1,e_2, \ldots e_n$ such that $e_i^2=1$ and $e_i e_j= -e_j e_i$ for all $i,j$. Then we have the following:

\begin{thm}\label{graded Bott map}
There is a $\ast$-homomorphim $b\colon S \rightarrow \Sigma \widehat{\otimes} \mathbb{F}_{1,0}$ inducing an isomorphism
\[E(A,B) \cong E(A, \Sigma B \widehat{\otimes} \mathbb{F}_{1,0}) .\]
\end{thm}

Also for Section~\ref{$K$-theory}, we have the following definitions coming from the $E$-theory definition. The $K$-theory groups are given by: 
$$K_n(A) = [\mathcal{S}, \Sigma^n A \widehat{\otimes} \mathcal{K}(\mathcal{H})] \cong E_n(\mathbb{R}, \Sigma^n A),$$ 
and the $K$-homology is given by
$$\mathbb{K}_{\text{hom}}(A) = E_n(A, \mathbb{R}).$$


\section{Quasi-topological spaces}
In this section we generalise the notion of orthogonal spectra to quasi-topological spaces. In order to do this, we have to define the notion of a quasi-continuous group action and further prove we have a symmetric monoidal structure on the category of quasi-orthogonal sequences which we also define. 

Firstly we recall some notions of quasi-topological spaces by Spanier~\cite{Spa63} which we will need as it is not known if we can put a standard topology on the set of asymptotic morphisms. 
\begin{definition}
A \emph{quasi-topology} on a set $X$, is a collection of sets of maps from $C$ to $X$ for each compact Hausdorff space $C$, written $Q(C,X)$, called quasi-continuous and satisfying:
\begin{itemize}
\item any constant map $C \rightarrow X$ belongs to $Q(C,X)$, 
\item if $f \colon C_1 \rightarrow C_2$ is a map of compact Hausdorff spaces and $g \in Q(C_2,X)$ then $gf \in Q(C_1,X)$,  
\item for a disjoint union $C = C_1 \amalg C_2$ of closed compact Hausdorff spaces, a map $g \colon C \rightarrow X$ is contained in $Q(C,X)$ if and only if $g|_{C_i} \in Q(C_i,X)$ for $i=1,2$,
\item for every $f \colon C_1 \rightarrow C_2$ surjective map of compact Hausdorff spaces, then a map $h \colon C_2 \rightarrow X$ is quasi-continuous if $h \circ f$ is quasi-continuous. 
\end{itemize}
A \emph{quasi-topological space} is a set $X$ endowed with a quasi-topology as described above. 
\end{definition}
If $X$ is a topological space we can obtain a quasi-topology on $X$ by considering $Q(C,X)$ as the set of continuous maps from $C$ to $X$ in the topology of $X$. 

A map of quasi-topological spaces $f \colon X \rightarrow Y$ is called \emph{quasi-continuous} if $g \in Q(C,X)$ implies that the composite $fg \in Q(C,Y)$. Also by the definition of quasi-continuous maps, a composite of quasi-continuous maps is also quasi-continuous. 
A \emph{quasi-homeomorphism} $f \colon X \rightarrow Y$ between quasi-topological spaces is a quasi-continuous bijection with a quasi-continuous inverse $g \colon Y \rightarrow X$.

\begin{definition} Let $X,Y$ be quasi-topological spaces and $f,g \colon X \rightarrow Y$ be quasi-continuous maps. Then a homotopy is a quasi-continuous map 
\[H \colon X \rightarrow C([0,1],Y),\]
such that for all $x \in X$, $H(x)(0) = f(x)$ and $H(x)(1)= g(x)$. 
\end{definition}

The suspension and loop space of a quasi-topological space $X$ are defined similarly to the case of topological spaces by
\[\Sigma_{\text{top}}X = S^1 \wedge X,\]
and  
\[\Omega X = \{ \mu \colon S^1 \rightarrow X ~ | ~ \mu ~\text{is quasi-continuous and basepoint preserving} \}\]
and we consider the circle $S^1$ with the quasi-topology that comes from the standard topology on $\mathbb{R}^2$. That is, our quasi-continuous maps are the continuous maps from every compact Hausdorff space in to $S^1$ in the topology from $\mathbb{R}^2$. 
Now we check that $\Sigma_{\text{top}}$ and $\Omega$ are adjoints in the category where objects are quasi-topological spaces and arrows are quasi-continuous maps. In order to do this, we consider an abstract result and obtain it as a corollary. 

\begin{proposition}\label{wedgesarefunctions}
Let $X,Y$ and $Z$ be quasi-topological spaces. Then 
\[F(X \wedge Y, Z) ~\text{and}~ F(X, F(Y,Z)),\]
are quasi-homeomorphic. 
\end{proposition}
\begin{proof}
We define $\alpha \colon F(X \wedge Y, Z) \rightarrow F(X, F(Y,Z))$ by 
\[((\alpha(f))(x))(y) = f(x \wedge y),\] 
where $f \in F(X \wedge Y, Z)$, $x \in X$ and $y \in Y$. 
Then $\alpha$ is quasi-continuous since $f$ is quasi-continuous. 

Define $\beta \colon F(X, F(Y,Z)) \rightarrow F(X \wedge Y, Z)$ by 
\[(\beta(g))(x \wedge y)=(g(x))(y),\]
where $g \in F(X, F(Y,Z))$, $x \in X$ and $y \in Y$.
So $\beta$ is quasi-continous since $g$ is quasi-continuous. 

Finally $\alpha$ and $\beta$ are inverses, so we obtain a quasi-homeomorphism.
\end{proof}

\begin{corollary}\label{sigmaandomegaadjointinqts}
$\Sigma_{\text{top}}$ and $\Omega$ are adjoints in the category of quasi-topological spaces. That is 
\[F(\Sigma_{\text{top}}X,Y) ~ \text{and} ~ F(X, \Omega Y)\]
are quasi-homeomorphic. 
\end{corollary}
\begin{proof}
This follows from Proposition~\ref{wedgesarefunctions} since 

\[F(\Sigma_{\text{top}}X,Y) = F(X \wedge S^1, Y),\]
and
\[F(X, \Omega Y) = F(X, F(S^1, Y)).\]
\end{proof}

Let $A,B$ be $C^\ast$-algebras. 
Further by work of Dardalat-Meyer~\cite{DM12} we can define a quasi topology on $\text{Asy}(A,B)$, the set of asymptotic morphisms from $A$ to $B$. 
We define the set of quasi-continuous maps from a compact Hausdorff space $Y$ to $\text{Asy}(A,B)$ to be the $\text{Asy}(A, C(Y,B))$, mentioned in~\cite{DM12}. That is, more precisely we have 
\begin{definition}\label{quasi-topology on Asy}
For a compact Hausdorff space $Y$, a map $h \colon Y \rightarrow \text{Asy}(A,B)$ is quasi-continuous when for each $t \in [1, \infty)$ the map $\widetilde{h_t}(a) \colon Y \rightarrow B$ defined by $$\widetilde{h_t}(a)(y) = h(y)_t(a),$$ is continuous. 
\end{definition}
Now we check that this is a quasi-topology. 

\begin{proposition}
The set of asymptotic morphisms from $A$ to $B$, $\text{Asy}(A,B)$, is a quasi-topological space when equipped with the above quasi-topology.
\end{proposition}
\begin{proof}
We must check the axioms. 
Let $c \colon Y \rightarrow \text{Asy}(A,B)$ be constant. Then for $y \in Y$, $c(y) =f_t \colon A \dashrightarrow B$ for a fixed $f$. Then we need to show that $c$ is quasi-continuous. That is to show for each map $c$, the map $\widetilde{c_t}(a) \colon Y \rightarrow B$ is continuous.
Then we define this map for $a \in A, y \in Y$ by
\[\widetilde{c_t}(a)(y) = c(y)_t(a) = f_t(a),\]
which is continuous as a function of $Y$ and hence $c$ is continuous.

Now let $f \colon Y_1 \rightarrow Y_2$ be a map of compact Hausdorff spaces and let $g \colon Y_2 \rightarrow \text{Asy}(A,B)$ be quasi-continuous. Then we want to show that $g f \colon Y_1 \rightarrow \text{Asy}(A,B)$ is quasi-continuous. That is, we need to show that $\widetilde{gf}_t(a) \colon Y_1  \rightarrow B$ is continuous. Now $f$ is continuous and since $g$ is quasi-continuous, we have each $\widetilde{g}_t(a) \colon Y_2 \rightarrow B$ is continuous and 
\[ \widetilde{g}_t(a)(y) = g(y)_t(a).\]
Now, 
\[\widetilde{gf}_t(a)(y) = \widetilde{g}_t(a)f(y),\]
which is continuous in $Y$ so is $\widetilde{gf}_t(a)$ is continuous, yielding that $gf$ is quasi-continuous. 

Let $Y = Y_1 \amalg Y_2$ of compact Hausdorff spaces. Then we need to show that $g \colon Y \rightarrow \text{Asy}(A,B)$ is quasi-continuous if and only if $g|_{Y_i}$ is quasi-continuous for $i=1,2$. 
Suppose $g \colon Y \rightarrow \text{Asy}(A,B)$ is quasi-continuous. Then map $\widetilde{g_t}(a) \colon Y \rightarrow B$ defined by $$\widetilde{g_t}(a)(y) = g(y)_t(a),$$ is continuous.
Now by properties of continuous functions we know that the restriction of a continuous function is continuous. 
Suppose that $g|_{Y_i}$ is quasi-continuous, then ${(\widetilde{g}|_{Y_i})}_t(a) \colon Y_i \rightarrow B$ is continuous. Then by properties of continuous functions we know that if the restrictions are continuous then the map of the disjoint uniion will be continuous. 

Finally, we need to check for every surjective map $f \colon Y_1 \rightarrow Y_2$ of compact Hausdorff spaces that $g \colon Y_2 \rightarrow \text{Asy}(A,B)$ is quasi-continuous if $g f \colon Y_1 \rightarrow \text{Asy}(A,B)$ is quasi-continuous. 
Then for a particular map $f$ and by the above argument
\[\widetilde{gf}_t(a)(y) = \widetilde{g}_t(a)f(y),\]
is continuous. 
Let $f \colon Y_1 \rightarrow Y_2$ be the identity, then 
\[\widetilde{gf}_t(a)(y) = \widetilde{g}_t(a)f(y) =\widetilde{g}_t(a),\]
and hence $\widetilde{g}_t(a)$ is continuous and the result follows. 
So we do indeed have a quasi-topology on $\text{Asy}(A,B)$ defined as above.
\end{proof}

For a quasi-topological space $X$, let $X_+$ denote the space with a basepoint. 

\begin{proposition}
The quasi-topological spaces $\Omega\text{Asy}(A,B)$ and $\text{Asy}(A, \Sigma B)$ are quasi-homeomorphic. 
\end{proposition}
\begin{proof}
By the definition of a quasi-topology, we know that
\[\Omega \text{Asy}(A,B) \colon = Q(S^1, \text{Asy}(A,B))_+.\]

Then by the definition~\ref{quasi-topology on Asy}, we have that 
\[Q(Y, \text{Asy}(A,B))_+ = \text{Asy}(A, C(Y,B)_+),\]
and then that 
\begin{align*}
\Omega \text{Asy}(A,B) 
& = Q(S^1, \text{Asy}(A,B))_+ \\
& = \text{Asy}(A, C(S^1,B)_+) \\
& = \text{Asy}(A, \Sigma B). \\
\end{align*}
\end{proof}

The above results hold in the case of graded asymptotic morphisms. 
\subsection{Group actions}
The following definition makes sense since a topological group can be viewed as a quasi-topological group. 

\begin{definition}
Let $G$ be a topological group acting on a quasi-topological space $X$. Then the group action is called \emph{quasi-continuous} if the map $G \times X \rightarrow X$ is quasi-continuous. If this is the case we say that the set $X$ is a \emph{quasi $G$-space}.
\end{definition}

\begin{definition}
A map $f \colon X \rightarrow Y$ of quasi $G$-spaces is a \emph{quasi $G$-map} if it is $G$-equivariant. That is, for all $g \in G$, we have 
\[f(gx) = g(f(x)).\]
\end{definition}

Now we consider basepoint preserving group actions. 

\begin{proposition}\label{groupactionpreservebpt}
Let $G$ and $H$ be groups. 
Let $X$ be a quasi $G$-space, $Y$ a quasi $H$-space where the group actions preserve the basepoints of both $X$ and $Y$. Then there are basepoint preserving actions of $G \times H$ on $X \times Y$, $X \vee Y$ and $X \wedge Y$. These actions are defined in the obvious way.
\end{proposition}
%
%
%

We now need the notion of a balanced smash product. 
\begin{definition}
Let $X$ be a right quasi $G$-space and $Y$ a left quasi $G$-space, then we can from the \emph{balanced smash product} $X \wedge_G Y$, which is the quotient space $X \wedge Y/ \sim_G$ where 
\[(xg \wedge y) \sim_G (x \wedge g^{-1}y) \Leftrightarrow (x \wedge y) \sim_G (xg \wedge g^{-1}y),\]
for all $g \in G$.
\end{definition}

Let the equivalence class of $x \wedge y$ be denoted by $x \wedge_G y$. Now using these we can construct a left quasi $G$-space. 

Let $G$ be a topological group and $H$ a subgroup. Then let $X$ be a based left quasi $H$-space where $G$ acts by preserving the basepoint. Let $G_+ = G \amalg \{\ast\}$, then we can construct the right quasi $G$-space denoted $G_+ \wedge_H X$ using the above equivalence classes. Additionally we can actually define a left quasi $G$-action on this space by the following map: 
\[(f, g \wedge_H x) \mapsto fg \wedge_H x,\]
for all $ f \in G$. 

To prove this is well-defined action it is a formality of using the fact that $H$ is a subgroup of $G$. 

Let $X$ and $Y$ be based quasi $G$-spaces. Then let $Q_G(X,Y)$ denote the set of basepoint preserving quasi $G$-maps. 
Then we have the following result:

\begin{proposition}\label{g-maps=h-maps}
Let $H$ be a subgroup of a group $G$.
Let $X$ be a left quasi $H$-space and $Y$ a left quasi $G$-space. There there is a natural bijection 
\[Q_H(X,Y) \longleftrightarrow Q_G(G_+ \wedge_H X, Y).\] 
\end{proposition}
\begin{proof}
We define $\alpha \colon Q_H(X,Y) \rightarrow Q_G(G_+ \wedge_H X, Y)$.
Let $f \in Q_H(X,Y)$ and $g \wedge_H x \in G_+ \wedge_H X$, then we define 
\[\alpha(f)(g \wedge_H x) = gf(x),\]
and we then need to check that $\alpha$ is well-defined and also $G$-equivariant. Since $(g \wedge_H x) \sim (gh \wedge_H h^{-1} x)$, then 
\begin{align*}
\alpha(f)(gh \wedge_H h^{-1} x) 
& = ghf(h^{-1}x)  = ghh^{-1}f(x) ~ ~\phantom{sh} \text{since f is a quasi H-map} \\
& = gf(x)  = \alpha(f)(g \wedge_H x).\\
\end{align*} 
Then $\alpha(f)$ is $G$-equivariant since 
\[\alpha(f)(g'g \wedge_H x) = g' g f(x) = g'(gf(x)) = g' \alpha(f)(g \wedge_H x),\]
for all $g' \in G$.

Now define $\beta \colon Q_G(G_+ \wedge_H X, Y) \rightarrow Q_H(X,Y)$. Let $k \in Q_G(G_+ \wedge_H X, Y)$, $x \in X$, then
\[ \beta(k)(x) = k(e \wedge_H x),\]
where $e$ denotes the identity in $G$. 
Then $\beta$ is $H$-equivariant since 
\begin{align*}
\beta(k)(hx) 
& = k(e \wedge_H hx) \\
&= k(h^{-1}h \wedge_H hx) \\
& = k(h \wedge_H x) \phantom{s} \text{by equivalence relations}\\
& = hk(e \wedge_H x) \phantom{s} \text{as} ~ k ~ \text{is a}~ H\text{-map}\\
& = h\beta(k)x.
\end{align*}
Then is is clear that $\alpha$ and $\beta$ are inverse maps, and both are natural in $X$ and $Y$, so the result follows.
\end{proof}

\subsection{Quasi-Orthogonal sequences}
Let $\mathscr{O}$ be the \emph{category of finite dimensional real Euclidean inner product spaces and linear isometric isomorphisms} where we have objects to be the set 
\[\text{obj}(\mathscr{O}) =\{\mathbb{R}^n ~ | ~ n=0,1, \ldots \}\]
and morphisms are 
\begin{equation*}
\mathscr{O}(A,B) = 
\begin{cases}
O(n),& \text{if}\; A=B=\mathbb{R}^n \\
\emptyset ,& \text{otherwise}.
\end{cases}
\end{equation*}
It should be noted that this is a small category since the collection of objects is a set. 

Let $\mathscr{T}$ denote the \emph{category of quasi-topological spaces with basepoints and quasi-continuous maps}. So $\text{obj}(\mathscr{T})$ is the collection of quasi-topological spaces with basepoints and the morphisms $\mathscr{T}(X,Y)$ are the set of basepoint preserving quasi-continuous maps from $X$ to $Y$. 

Then we can obtain the product category $\mathscr{T} \times \mathscr{T}$ where 
$\text{obj}(\mathscr{T} \times \mathscr{T})$ are pairs $(X,Y)$ of quasi-topological spaces with basepoints, and morphisms are 
\[(\mathscr{T} \times \mathscr{T})((X,Y)(Z,W)) = \{ (f,g) ~ | ~ f \in \mathscr{T}(X,Z), g \in \mathscr{T}(Y,W)\} .\]

\begin{proposition}
The smash product $\wedge \colon \mathscr{T} \times \mathscr{T} \rightarrow \mathscr{T}$ of quasi-topological spaces is a functor.
\end{proposition}
\qed

The following definition of a quasi-orthogonal sequence is going to form part of the definition of a orthogonal quasi-spectrum. 

\begin{definition}
Let $\mathscr{O}$ and $\mathscr{T}$ be the categories defined above. Then we define the \emph{category of quasi orthogonal sequences} formed as the functor category $\mathscr{T}^\mathscr{O}$ with objects 
\[\text{obj}(\mathscr{T}^\mathscr{O}) = \{\text{functors} ~X \colon \mathscr{O} \rightarrow \mathscr{T} ~ | ~ X_n : = X(\mathbb{R}^n)\},\]
together with a left quasi-continuous basepoint preserving action of $O(n)$ on each $X_n$ for all $n\geq 0$,
and morphisms
\[\mathscr{T}^\mathscr{O}(X,Y) = \{ \varphi \colon X \rightarrow Y ~ |~ \varphi ~\text{is a natural transformation}\},\]
and such that a natural transformation is formed of sets of quasi-continuous basepoint preserving maps $\varphi_n \colon X_n \rightarrow Y_n$ that are $O(n)$-equivariant for $n \geq 0$, or equivalently that the map $\varphi_n$ commutes with the group action of $O(n)$ on $X_n$ and $Y_n$.
\end{definition}

A useful example of such a functor category will be the unit sequence coming from the orthogonal sequence defined below. 
Consider a based topological space $K$, then define the orthogonal sequence with $n$-space:
\begin{equation*}
(G_p K)_n =
\begin{cases}
O(n)_+ \wedge K, & \text{if} ~n=p \\
\{\ast\}, & \text{otherwise} \\
\end{cases}
\end{equation*}
Then the \emph{unit sequence} is when we just have the topological space $S^0$, given by the sequence 
\[G_0 S^0 = \{S^0, \ast, \ast, \ldots \}.\]

We also consider quasi-biorthogonal sequences since they will help us in defining our smash product structure. 

The category of \emph{quasi-biorthogonal sequences} is defined to be the category with objects 
\[\text{obj}(\mathscr{T}^{\mathscr{O}\times \mathscr{O}}) = \{X \colon \mathscr{O} \times \mathscr{O} \rightarrow \mathscr{T} ~ | ~ X ~ \text{is a functor} \},\]
together with a quasi-continuous basepoint preserving left-action of $O(m) \times O(n)$, 
and 
\[\mathscr{T}^{\mathscr{O}\times \mathscr{O}}(X,Y) = \{ \psi \colon X \rightarrow Y ~ |~ \psi ~\text{is a natural transformation}\},\]
formed of sets of quasi-continuous basepoint preserving maps $\psi_{m,n} \colon X_{m,n} \rightarrow Y_{m,n}$, where $X_{m,n} :=X(\mathbb{R}^m, \mathbb{R}^n)$, that are $O(m) \times O(n)$-equivariant for all $m,n\geq 0$.

Using this we can define the external smash product of two quasi-orthogonal sequence $X$ and $Y$. 

\begin{definition}
Define the \emph{external smash product} $X \overline{\wedge} Y$ to be the quasi-biorthogonal sequence given by the composition
\[ \mathscr{O} \times \mathscr{O} \xrightarrow{X \times Y} \mathscr{T} \times \mathscr{T} \xrightarrow{\wedge} \mathscr{T},\]
defined by 
\[(X \overline{\wedge} Y)_{m,n} = (X \overline{\wedge} Y)(\mathbb{R}^m, \mathbb{R}^n) = X(\mathbb{R}^m) \wedge Y(\mathbb{R}^n) = X_m \wedge Y_n.\]
Then by Proposition~\ref{groupactionpreservebpt}, the quasi-topological space $X_n \wedge Y_m$ has a quasi-$O(n) \times O(m)$-action. 
\end{definition}

For a general quasi-orthogonal sequence $X$ we can define a quasi-biorthogonal sequence $X \circ \oplus$ by: 
\[(X \circ \oplus)_{m,n} = (X \circ \oplus )(\mathbb{R}^m, \mathbb{R}^n) = X(\mathbb{R}^{m+n}) = X_{m+n}.\]

Now we can construct the tensor product of quasi-orthogonal sequences since the category $\mathscr{T}$ is complete and cocomplete. 

\begin{definition}
For quasi orthogonal sequence $X$ and $Y$ we define the tensor product of $X$ and $Y$ to be the quasi-orthogonal sequence 
\[(X \otimes Y)_n = \bigvee_{p+q=n} O(n)_+ \wedge_{O(p) \times O(q)} X_p \wedge Y_q,\]
where we define the $O(n)$-action on $(X \otimes Y)_n$ by acting on each wedge summand. 
\end{definition}

Then we can combine the external smash product and tensor product of quasi-orthogonal sequences as a natural bijection:

\begin{proposition}\label{externalandtensorproduct}
For quasi-orthogonal sequences $X,Y$ and $Z$, there is a natural bijection
\[\mathscr{T}^{\mathscr{O} \times \mathscr{O}}(X \overline{\wedge} Y, Z \circ \oplus) \longleftrightarrow \mathscr{T}^\mathscr{O}(X \otimes Y, Z).\] 
\end{proposition}
\begin{proof}
Let $f \colon X \overline{\wedge} Y \rightarrow Z \circ \oplus$ be a natural transformation in the category of quasi-biorthogonal sequences. Then $f_{p,q}\colon X_p \wedge Y_q \rightarrow Z \circ \oplus$ is quasi $O(p) \times O(q)$-equivariant and then by proposition~\ref{g-maps=h-maps}, this corresponds to a quasi $O(n)$-equivariant map, with $n=p+q$
\[\overline{f}_{p,q} \colon O(n)_+ \wedge_{O(p) \times O(q)} X_p \wedge Y_q \rightarrow Z_n.\]
Now fixing $n$ and letting $p$ and $q$ vary, this allows us to obtain a quasi $O(n)$-equivariant map 
\[\overline{f}_n = \bigvee_{p+q=n} O(n)_+ \wedge_{O(p) \times O(q)} X_p \wedge Y_q \rightarrow Z_n,\]
which is a quasi-continuous basepoint preserving $O(n)$-equivariant map in $\mathscr{T}^{\mathscr{O}}$ from $X \otimes Y$ to $Z$.

Now we construct a map the other way. 
Let $g \in \mathscr{T}^\mathscr{O}(X \otimes Y, Z)$. Then $g$ is a wedge summand of basepoint preserving quasi-continuous $O(n)$-equivariant maps 
\[g_n \colon \bigvee_{p+q=n} O(n)_+ \wedge_{O(p) \times O(q)} X_p \wedge Y_q \rightarrow Z_n,\] 
for all $n \geq 0$.
Also, we can write that $g_n = \bigvee_{p+q=n} g_{p,q}$, where 
\[g_{p,q} \colon  O(n)_+ \wedge_{O(p) \times O(q)} X_p \wedge Y_q \rightarrow Z_n,\]
and by proposition~\ref{g-maps=h-maps}, we obtain a basepoint preserving quasi-continuous $O(p) \times O(q)$-equivariant map as required. 
\end{proof}

Let $G_0S^0$ be the unit quasi-orthogonal sequence \[(G_0 S^0)_n = (S^0, \ast,\ast,\ldots).\]

For details of the following see Chapter~4 of~\cite{Brothesis}.
\begin{proposition}
The category of quasi-orthogonal sequences forms a symmetric monoidal category $(\mathscr{T}^{\mathscr{O}}, \otimes, G_0S^0)$.
\end{proposition}
\qed

Let $S = (S^0, S^1, S^2 \ldots )$ be the quasi-orthogonal sequence defined in terms of quasi-topological spaces. 
For a proof of the subsequent result see Proposition 4.3.13~ in~\cite{Brothesis}.
\begin{proposition}\label{spheresequenceisacommonoid}
The orthogonal sequence of quasi-topological spaces $S = (S^0, S^1, S^2 \ldots )$ is a commutative monoid in the symmetric monoidal category $(\mathscr{T}^{\mathscr{O}}, \otimes, G_0S^0)$. 
\end{proposition}
%
%
\qed

\section{E-theory orthogonal quasi-spectra}
This section brings together ideas from the previous section, since we will define the notion of an orthogonal quasi-spectrum which is a quasi-orthogonal sequence with added structure. We will show that we have an orthogonal quasi-spectrum representing the graded $E$-theory groups and thereafter show we have a smash product.

\subsection{Quasi-Spectra}

We begin by defining concepts we have seen before in terms of quasi-topological spaces. 

A \emph{quasi-spectrum} is a sequence of based quasi-topological spaces $X_0, X_1, \ldots$ with structure maps $ \epsilon \colon X_m \rightarrow \Omega X_{m+1}$ that are quasi-continuous. 
An \emph{$\Omega$-quasi-spectrum} is a quasi-spectrum where for all natural numbers $m$ the structure maps $\epsilon \colon X_m \rightarrow \Omega X_{m+1}$ are weak equivalences. 
Then we can define an orthogonal quasi-spectrum:

\begin{definition}
An \emph{orthogonal quasi-spectrum} is 
\begin{itemize}
\item a sequence of based quasi-topological spaces $X_0, X_1, \ldots$ 
\item a basepoint preserving quasi-continuous left action of $O(m)$ on each $X_m$ for all $m$, and 
\item a collection of based structure maps $\sigma=\sigma_m \colon X_m \wedge S^1 \rightarrow X_{m+1}$ that are quasi-continuous, 
\end{itemize}
such that for each $m,n \geq 0$, the iterated map 
\[{\sigma^n_m} \colon X_m \wedge S^n \rightarrow X_{m+1} \wedge S^{n-1} \rightarrow \ldots  \rightarrow X_{m+n},\]
is quasi-continuous and $O(m) \times O(n)$-equivariant.  
\end{definition}
In the same manner, we have that a morphism of orthogonal quasi-spectrum $f \colon X \rightarrow Y$ is a collection of quasi-$O(m)$-equivariant maps $f_m \colon X_m \rightarrow Y_m$ for all $m$, which satisfy the following commutative diagram: 
\[\xymatrixcolsep{3pc}\xymatrixrowsep{3pc}\xymatrix{
X_m \wedge S^1 \ar[d]_-{\sigma_m} \ar[r]^-{f  \wedge \text{id}_{S^1}} &  Y_m \wedge S^1 \ar[d]^-{\sigma_m}\\
X_{m+1} \ar[r]^-{f_{m+1}} & Y_{m+1},}\]
or alternatively that the following diagram commutes: 
\[\xymatrixcolsep{3pc}\xymatrixrowsep{3pc}\xymatrix{
X_m \ar[d]_-{\epsilon_m} \ar[r]^-{f_m} &  Y_m \ar[d]^-{\epsilon_m}\\
\Omega X_{m+1} \ar[r]^-{\Omega f_{m+1}} & \Omega Y_{m+1}.}\]
It is easily seen that any orthogonal spectrum is an orthogonal quasi-spectrum. 
By Corollary~\ref{sigmaandomegaadjointinqts}, the structure maps in the definition of quasi-spectrum can be defined in terms of loop spaces. 
Notice that an orthogonal quasi-spectrum is a quasi-orthogonal sequence with  more structure. 

\begin{proposition}
The category of right $S$-modules, mod-$S$ is naturally equivalent to the category of orthogonal quasi-spectrum.
\end{proposition}
\begin{proof}
Consider the multiplication map $\nu \colon X \otimes S \rightarrow S$ for a right $S$-module $X$. Then by Proposition~\ref{externalandtensorproduct} we have a set of $O(m) \times O(n)$-equivariant maps 
\[ \nu_m^n \colon X_m \wedge S^n \rightarrow X_{m+n},\]
for $m,n \geq 0$ with unit quasi-homeomorphism $\nu_m^0$.
Now this action is associative so it follows that the structure maps are then defined by $\nu_m$. 

Conversely, consider the set of structure maps 
\[ \sigma_p^n \colon X_p \wedge S^p \rightarrow X_{n+p},\]
for a spectrum $X$ and $p,n \geq 0$, with unit quasi-homeomorphism $\sigma_p^0$. Then we have a multiplicative map $\nu \colon X \otimes S \rightarrow X$ defining a right $S$-module. 
Since these constructions are inverses, we have a natural equivalence of  these two categories. 
\end{proof}

Hence we can obtain a tensor product of orthogonal quasi-spectrum since we have a tensor product in the category of right $S$-modules.

\begin{definition}
Let $\mathbb{X}$ be an orthogonal quasi-spectrum with spaces $X_n$. For each integer $k \in \mathbb{Z}$ we define the \emph{$k$-th stable homotopy group} $\pi_k(\mathbb{X})$ to be the direct limit
\[\pi_k(\mathbb{X}) = \varinjlim_n \pi_{k+n} X_n,\]
under the maps $\epsilon_{\ast} \colon \pi_{k+n} X_n \rightarrow \pi_{k+n+1} X_{n+1}$ induced from the structure maps $\epsilon \colon \Omega^{k+n} X_n \rightarrow \Omega^{k+n+1} X_{n+1}$. 
\end{definition}

\subsection{Graded $E$-theory Spectra}

Let $\text{Asy}_g(A,B)$ denote the set of graded asymptotic morphisms from $A$ to $B$ with the quasi-topology as defined in Definition~\ref{quasi-topology on Asy}.
\begin{proposition}\label{tensorwithanasymptoticmorphismisqc}
The map of quasi-topological spaces $$f \colon \text{Asy}_g(A,B) \rightarrow \text{Asy}_g(D \widehat{\otimes} A, D \widehat{\otimes} B)$$ defined by 
\[f(x_t) = \text{id}_D \widehat{\otimes} x_t,\]
is quasi-continuous for all $x_t \in \text{Asy}_g(A,B)$.
\end{proposition}
\begin{proof}
Since gradings follow immediately, we consider ungraded asymptotic morphisms throughout the proof. 
To prove a map of quasi-topological spaces is quasi-continuous, wwe need to check that for a quasi-continuous map $g: Y \rightarrow \text{Asy}(A,B)$ where $Y$ is a compact Hausdorff space, that the composition $fg \colon Y \rightarrow \text{Asy}(D \widehat{\otimes} A, D \widehat{\otimes} B)$ is quasi-continuous.
Suppose $g: Y \rightarrow \text{Asy}(A,B)$ where $Y$ is a compact Hausdorff space is quasi-continuous. 
Then by definition~\ref{quasi-topology on Asy} we know that $g$ is quasi-continuous when for each $t \in [1,\infty)$ the map $\widetilde{g_t}(a) \colon Y \rightarrow B$ defined by 
$$\widetilde{g_t}(a)(y) = g(y)_t(a)$$
is continuous. 
Then we define 
$fg \colon Y \rightarrow \text{Asy}(D \widehat{\otimes} A, D \widehat{\otimes} B)$ by for each $t \in [1,\infty)$ 
\[f(g(y)_t)(a) = \text{id}_D \widehat{\otimes} g(y)_t(a) =\text{id}_D \widehat{\otimes} \widetilde{g_t}(a)(y)\] 
but since $g$ is quasi-continuous and that 
\[f(g(y)_t)(a) = \widetilde{fg}_t (a)(y),\]
it follows from the definition of quasi-topology on the set of asymptotic morphisms that $fg$ is quasi-continuous.
\end{proof}

\begin{definition}\label{E-theory spectrum}
Let $\mathcal{K} = \mathcal{K}(\mathcal{H})$.
Define $\mathbb{X}(A,B)$ to be the sequence of based quasi-topological spaces $$X_m = \text{Asy}_g (\mathcal{S} \widehat{\otimes} A \widehat{\otimes} \mathcal{K}, B \widehat{\otimes} \mathbb{F}_{m,0} \widehat{\otimes} \mathcal{K})$$ where $m \geq 0$. Define maps $\epsilon_m \colon X_m \rightarrow \Omega X_{m+1}$:
  \[\xymatrix{\text{Asy}_g( \mathcal{S} \widehat{\otimes} A \widehat{\otimes} \mathcal{K}, B \widehat{\otimes} \mathbb{F}_{m,0} \widehat{\otimes} \mathcal{K})  \ar[r]  & \Omega \text{Asy}_g(\mathcal{S} \widehat{\otimes} A \widehat{\otimes}\mathcal{K},  B \widehat{\otimes} \mathbb{F}_{m+1,0} \widehat{\otimes} \mathcal{K}) \ar@{=}[d]^{\cong} \\ &  \text{Asy}_g(\mathcal{S} \widehat{\otimes} A \widehat{\otimes} \mathcal{K}, \Sigma ( B \widehat{\otimes} \mathbb{F}_{m+1,0}) \widehat{\otimes} \mathcal{K})}   \]
by: 
\[\epsilon (x_t) = (b \widehat{\otimes} \text{id}_{B \widehat{\otimes} \mathbb{F}_{m,0} \widehat{\otimes} \mathcal{K}}) \circ (\text{id}_{\mathcal{S}} \widehat{\otimes} x_t)  \circ (\Delta \widehat{\otimes} \text{id}_{A \widehat{\otimes}\mathcal{K}}),\]
for all $x_t \in\text{Asy}_g (\mathcal{S} \widehat{\otimes} A \widehat{\otimes} \mathcal{K}, B \widehat{\otimes} \mathbb{F}_{m,0} \widehat{\otimes} \mathcal{K})$ and the Bott map $b \in \text{Hom}_g(\mathcal{S},\Sigma  \mathbb{F}_{1,0})$ from Theorem~\ref{graded Bott map}. 
Alternatively, we also have a map $\sigma_m \colon X_m \wedge S^1 \rightarrow X_{m+1}$ defined by
\[\sigma_m(x_t,s) = \epsilon_m(x_t)(s),\]
with $x_t \in\text{Asy}_g (\mathcal{S} \widehat{\otimes} A \widehat{\otimes} \mathcal{K}(\mathcal{H}), B \widehat{\otimes} \mathbb{F}_{m,0} \widehat{\otimes} \mathcal{K}(\mathcal{H}))$ and $s \in S^1$. 
\end{definition}
\begin{definition}\label{actiononAsy}
We define a quasi-continuous action of the group $O(m)$ on the space $X_m$ as follows. 
First we consider the alternative definition of the Clifford algebra $\mathbb{F}_{m,0}$ as $\text{Cliff}(V)$.  Recall that for $V$ an $m$-dimensional Euclidean vector space, $\text{Cliff}(V) = G(V)/\sim$ where $G(V)$ is the algebra generated by $V$ subject to the equivalence relation $\sim$ defined by 
\[ v^2 = ||v||^2 \cdot 1\]
for all $v \in V$. 
We write $ab$ for the product of two elements $a,b \in \text{Cliff}(V)$.

If $V = \mathbb{R}^m$, then we have a natural group action $(H,v) \mapsto Hv$ where $H \in O(m)$, $v \in V$.

Then we can define a group action of $O(m)$ on $G(V)$ by 
\[H(v_1 \ldots v_k) \mapsto H(v_1) \ldots H(v_k) ~ ~ \text{and}~ H(1) =1\]
for all $H \in O(m)$. Then this gives a group action of $O(m)$ on $\text{Cliff}(V)$ since 
\begin{align*}
H(v^2) 
&= H(v) H(v) = (H(v))^2 \\
& = ||H(v)||^2 \cdot 1 = ||v|| \cdot 1 ~ \text{since} ~  H ~ \text{is orthogonal}. \\ 
\end{align*}
So then we get a group action
\[\lambda \colon O(m) \times \mathbb{F}_{m,0} \rightarrow \mathbb{F}_{m,0},\]
by \[\lambda(H, (e_1,e_2, \ldots e_m)) =  H(e_1)H(e_2) \ldots H(e_m),\]
where $H \in O(m)$, $e_1,e_2, \ldots e_m$ are the generators of the algebra $\mathbb{F}_{m,0}$.
Then we define 
\[\lambda_{\ast} \colon O(m) \times B \widehat{\otimes} \mathbb{F}_{m,0} \widehat{\otimes} \mathcal{K}(\mathcal{H}) \rightarrow B \widehat{\otimes} \mathbb{F}_{m,0} \widehat{\otimes} \mathcal{K}(\mathcal{H})\]
by \[\lambda_{\ast}(H, b \widehat{\otimes} x \widehat{\otimes} p) = b \widehat{\otimes} \lambda(H,x) \widehat{\otimes} p \]
with $H \in O(m)$, $b \in B$, $x \in \mathbb{F}_{m,0}$ and $p \in \mathcal{K}(\mathcal{H})$.
Then we finally define a group action of $O(m)$ on $X_m$
\[\lambda_{\ast \ast} \colon O(m) \times \text{Asy}_g (\mathcal{S} \widehat{\otimes} A \widehat{\otimes} \mathcal{K}(\mathcal{H}), B \widehat{\otimes} \mathbb{F}_{m,0} \widehat{\otimes} \mathcal{K}(\mathcal{H}))\] \[\longrightarrow \text{Asy}_g (\mathcal{S} \widehat{\otimes} A \widehat{\otimes} \mathcal{K}(\mathcal{H}), B \widehat{\otimes} \mathbb{F}_{m,0} \widehat{\otimes} \mathcal{K}(\mathcal{H})),\]
by 
\[\lambda_{\ast \ast} (H, \alpha_t)(x) = \lambda_{\ast}(H, \alpha_t(x)),\]
where we have $\alpha_t \in \text{Asy}_g(\mathcal{S} \widehat{\otimes} A \widehat{\otimes} \mathcal{K}(\mathcal{H}), B \widehat{\otimes} \mathbb{F}_{m,0} \widehat{\otimes} \mathcal{K}(\mathcal{H}))$, $x \in \mathcal{S} \widehat{\otimes} A \widehat{\otimes} \mathcal{K}(\mathcal{H})$ and $H \in O(m)$. Then it follows that this action is $O(m)$-equivariant.
\end{definition}

The following is true by Proposition~\ref{tensorwithanasymptoticmorphismisqc}.
\begin{proposition}\label{actiononAsyisqcetc}
The action in the previous definition is a basepoint preserving quasi-continuous action of $O(m)$ on $X_m$.
\end{proposition}
\qed
\begin{proposition}\label{structuremapisqc}
The map $\epsilon_m \colon X_m \rightarrow \Omega X_{m+1}$ is quasi-continuous and hence the map $\sigma_m \colon X_m \wedge S^1 \rightarrow X_{m+1}$ is quasi-continuous. 
\end{proposition}
\begin{proof}
Since both $$b \widehat{\otimes} \text{id}_{B \widehat{\otimes} \mathbb{F}_{m,0} \widehat{\otimes} \mathcal{K}(\mathcal{H})} ~~~\text{and}~~~ \Delta \widehat{\otimes} \text{id}_{A \widehat{\otimes}\mathcal{K}(\mathcal{H})}$$ are $\ast$-homomorphisms, these two maps are continuous. So it suffices to check the map $(\text{id}_{\mathcal{S}} \widehat{\otimes} x_t)$ is quasi-continuous since the a composition of continuous and quasi-continuous maps yields a quasi-continuous map. By proposition~\ref{tensorwithanasymptoticmorphismisqc}, with $D = \mathcal{S}$ it follows that the map $(\text{id}_{\mathcal{S}} \widehat{\otimes} x_t)$ is quasi-continuous. Hence the map $\epsilon_m$ is quasi-continuous. 
Since $\sigma_m$ is defined in terms of $\epsilon_m$ it is also quasi-continuous. 
\end{proof}

We define the iterated map $\sigma^n_m \colon X_m \wedge S^n \rightarrow X_{m+n}$ by 
\[\sigma^n_m(x_t, s_1, s_2, \ldots s_n) = \epsilon^n(x_t)(s_1)(s_2)\ldots(s_n),\]
where $x_t \in\text{Asy}_g (\mathcal{S} \widehat{\otimes} A \widehat{\otimes} \mathcal{K}(\mathcal{H}), B \widehat{\otimes} \mathbb{F}_{m,0} \widehat{\otimes} \mathcal{K}(\mathcal{H}))$ and $s_1, s_2, \ldots s_n$  is contained in $S^1\wedge S^1 \wedge \ldots \wedge  S^1$.

\begin{proposition}\label{mapisqcandequivariant}
The iterated map $\sigma^n_m \colon X_m \wedge S^n \rightarrow X_{m+n}$ defined above quasi-continuous and $O(m) \times O(n)$-equivariant.
\end{proposition}
\begin{proof}
By other results it suffices to check that the map is $O(m) \times O(n)$-equivariant. 

Firstly it is clear that $X_m \wedge S^n$ and $X_{m+n}$ are quasi $O(m) \times O(n)$-spaces. Let $i \colon O(m) \times O(n) \rightarrow O(m+n)$ be the inclusion map.

Now $O(m+n)$ acts on $\text{Asy}_g (\mathcal{S} \widehat{\otimes} A \widehat{\otimes} \mathcal{K}(\mathcal{H}), B \widehat{\otimes} \mathbb{F}_{m+n,0} \widehat{\otimes} \mathcal{K}(\mathcal{H}))$ by 
$$J(x_t) = (\text{id}_B \widehat{\otimes} J \widehat{\otimes} \text{id}_{\mathcal{K}(\mathcal{H})}) \circ x_t$$
for all $J \in O(m+n)$ and $x_t \colon \mathcal{S} \widehat{\otimes} A \widehat{\otimes} \mathcal{K}(\mathcal{H}) \rightarrow B \widehat{\otimes} \mathbb{F}_{m+n,0} \widehat{\otimes} \mathcal{K}(\mathcal{H})$. Here $J$ acts of $\mathbb{F}_{m+n,0}$ as defined earlier. 

Then $O(m) \times O(n)$ acts on $\text{Asy}_g(\mathcal{S} \widehat{\otimes} A \widehat{\otimes} \mathcal{K}(\mathcal{H}), B \widehat{\otimes} \mathbb{F}_{m,0} \widehat{\otimes} \mathcal{K}(\mathcal{H})) \wedge S^n$ by, 
$$(H,K)(x_t,s) = ((\text{id}_B \widehat{\otimes} H\widehat{\otimes} \text{id}_{\mathcal{K}(\mathcal{H})}) \circ x_t, Ks),$$ 
for all $H \in O(m)$, $K \in O(n)$, $x_t \in \text{Asy}_g (\mathcal{S} \widehat{\otimes} A \widehat{\otimes} \mathcal{K}(\mathcal{H}), B \widehat{\otimes} \mathbb{F}_{m,0} \widehat{\otimes} \mathcal{K}(\mathcal{H}))$ and $s \in S^n$.
Then we need to show that for $\sigma = \sigma^n_m \colon X_m \wedge S^n \rightarrow X_{m+n}$, that 
\[\sigma ((H,K)(x_t,s)) = i(H,K) \sigma (x_t,s), \]
that is, 
\[\sigma ((Hx_t,Ks) = i(H,K) \sigma (x_t,s).\]

That is to show, by definition of $\sigma$ that, 

\[\epsilon(Hx_t)( Ks) = i(H,K) \epsilon(x_t)(s).\]

Let $b^n = b \widehat{\otimes} \ldots \widehat{\otimes} b$ be the $n$-fold graded tensor product of the Bott map, $b \colon \mathcal{S} \rightarrow \Sigma \mathbb{F}_{1,0}$. 
Then $$b^n =b \widehat{\otimes} \ldots \widehat{\otimes} b \colon \mathcal{S}^n \rightarrow \Sigma^n \widehat{\otimes}\mathbb{F}_{n,0},$$
we have $b \widehat{\otimes} \ldots \widehat{\otimes} b(\lambda)(s) \in \mathbb{F}_{n,0}$ for $\lambda \in \mathcal{S}^n$ and $s \in S^n$. Then for $K \in O(n)$,
\[(b \widehat{\otimes} \ldots \widehat{\otimes} b)(\lambda)(Ks) = K(b \widehat{\otimes} \ldots \widehat{\otimes} b)(\lambda)(s).\]

Then by permuting copies of $\Sigma$ and extending by linearity we have an action of the orthogonal group.

Reconsidering \[\epsilon(Hx_t)( Ks) = i(H,K) \epsilon(x_t)(s),\]
the left hand side yields 
\[\epsilon(Hx_t)(Ks)= ((b^n \widehat{\otimes} \text{id}_{B \widehat{\otimes} \mathbb{F}_{m,0} \widehat{\otimes} \mathcal{K}(\mathcal{H})}) \circ (\text{id}_{\mathcal{S}} \widehat{\otimes} H x_t)  \circ (\Delta \widehat{\otimes} \text{id}_{A \widehat{\otimes}\mathcal{K}(\mathcal{H})})(Ks)),\]
and the right hand side yields 
$$i(H,K) \epsilon(x_t)(s) = i(H,K)(b \widehat{\otimes} \text{id}_{B \widehat{\otimes} \mathbb{F}_{m,0} \widehat{\otimes} \mathcal{K}(\mathcal{H})}) \circ (\text{id}_{\mathcal{S}} \widehat{\otimes} x_t)  \circ (\Delta \widehat{\otimes} \text{id}_{A \widehat{\otimes}\mathcal{K}(\mathcal{H})}).$$

Then 

\begin{eqnarray*}
\lefteqn{\epsilon(Hx_t)(Ks)} \\
& =& ((b^n \widehat{\otimes} \text{id}_{B \widehat{\otimes} \mathbb{F}_{m,0} \widehat{\otimes} \mathcal{K}(\mathcal{H})}) \circ (\text{id}_{\mathcal{S}} \widehat{\otimes} H x_t)  \circ (\Delta \widehat{\otimes} \text{id}_{A \widehat{\otimes}\mathcal{K}(\mathcal{H})}))(Ks) \\
&=& i(H,1)((b^n \widehat{\otimes} \text{id}_{B \widehat{\otimes} \mathbb{F}_{m,0} \widehat{\otimes} \mathcal{K}(\mathcal{H})}) \circ (\text{id}_{\mathcal{S}} \widehat{\otimes} x_t)  \circ (\Delta \widehat{\otimes} \text{id}_{A \widehat{\otimes}\mathcal{K}(\mathcal{H})})(Ks)) \\
&= & i(H,1) i(1,K)(b^n \widehat{\otimes} \text{id}_{B \widehat{\otimes} \mathbb{F}_{m,0} \widehat{\otimes} \mathcal{K}(\mathcal{H})}) \circ (\text{id}_{\mathcal{S}} \widehat{\otimes} x_t)  \circ (\Delta \widehat{\otimes} \text{id}_{A \widehat{\otimes}\mathcal{K}(\mathcal{H})})(s) \\ 
&=&  i(H,K)(b^n \widehat{\otimes} \text{id}_{B \widehat{\otimes} \mathbb{F}_{m,0} \widehat{\otimes} \mathcal{K}(\mathcal{H})}) \circ (\text{id}_{\mathcal{S}} \widehat{\otimes} x_t)  \circ (\Delta \widehat{\otimes} \text{id}_{A \widehat{\otimes}\mathcal{K}(\mathcal{H})})(s) \\ 
&=& i(H,K) \epsilon(x_t)(s)
\end{eqnarray*}
Then the result follows.

\end{proof}

The proof of the following result follows from the above propositions, namely Proposition~\ref{actiononAsyisqcetc}, Proposition~\ref{structuremapisqc} and Proposition~\ref{mapisqcandequivariant}.

\begin{proposition}
The spectrum $\mathbb{X}(A,B)$ is an orthogonal quasi-spectrum.
\end{proposition}
\qed

\begin{proposition}\label{isomorphismsinducedirectlimits}
If $G_0,G_1,G_2, \ldots$ is a sequence of groups with isomorphisms $\theta_n \colon G_n \rightarrow G_{n+1}$ for $n \geq 0$, then 
\[\varinjlim_n G_n = G_0.\] 
\end{proposition}
\begin{proof}
We first need to construct a commutative diagram.
    \[\xymatrix{G_n   \ar[dd]_{\theta_n} \ar[dr]^{\delta}&   \\ 
& G_0 \\ G_{n+1} \ar[ur]_{\psi} & }.\]
As $\theta_n$ is an isomorphism for all $n \geq 0$, we have inverses, so $\delta = (\theta_{0})^{-1} \ldots (\theta_{n-1})^{-1}$ and $\psi =  (\theta_0)^{-1} \ldots  (\theta_{n-1})^{-1} (\theta_n)^{-1}$ and hence the diagram commutes. 
Now we check that $G_0$ is unique. 
Suppose we have a group $H$ such that we have a group homomorphism $f\colon G_0 \rightarrow H$ which fits into the following diagram
    \[\xymatrix{ G_n \ar[dd]_{\theta_n} \ar[dr]_{\delta} \ar@/^/[drr]^{\mu_1} & & \\ 
&G_0 \ar@{.>}[r]|-{f} & H \\G_{n+1}\ \ar[ur]^{\psi} \ar@/_/[urr]_{\mu_{2}} & &}.\]
Then define $f= \mu_1 \delta^{-1}$ so our diagram commutes. Suppose that we have another group homomorphism $g \colon G_0 \rightarrow H$ fitting into the diagram. Then by commutativity we have $g \delta = \mu_1$, so $g = \mu_1 \delta^{-1} = f$ so $f$ is unique. 
\end{proof}
\begin{proposition}\label{directlimitofspectra2}
The direct limit $\varinjlim_n  E_g(A, \Sigma^{k+n} B \widehat{\otimes} \mathbb{F}_{n,0}) $ is $E_g(A, \Sigma^{k} B)$.
\end{proposition}
\begin{proof}
This result follows from Proposition~\ref{isomorphismsinducedirectlimits} where $$G_n = E_g(A, \Sigma^{k+n} B \widehat{\otimes} \mathbb{F}_{n,0})$$ and using Proposition~\ref{graded Bott map}. 
\end{proof}

\begin{proposition}
For all positive integers $k$, 
\[ \pi_k \mathbb{X}(A,B) = E_g(A, \Sigma^k B).\] 
\end{proposition}
\begin{proof}
Since $\mathbb{X}$ is an orthogonal quasi-spectrum we have that 
\[ \pi_k \mathbb{X}(A,B) = \varinjlim_n \pi_{k+n} E_n. \]
Then
\begin{align*}
 \varinjlim_n \pi_{k+n} X_n 
& =  \varinjlim_n \pi_0 \Omega^{k+n}\text{Asy}_g (\mathcal{S} \widehat{\otimes} A \widehat{\otimes} \mathcal{K}(\mathcal{H}), B \widehat{\otimes} \mathbb{F}_{n,0} \widehat{\otimes} \mathcal{K}(\mathcal{H})) \\
& = \varinjlim_n \pi_0 \text{Asy}_g (\mathcal{S} \widehat{\otimes} A \widehat{\otimes} \mathcal{K}(\mathcal{H}), \Sigma^{k+n}  B \widehat{\otimes} \mathbb{F}_{n,0} \widehat{\otimes} \mathcal{K}(\mathcal{H})) \\
& = \varinjlim_n  \llbracket \mathcal{S} \widehat{\otimes} A \widehat{\otimes} \mathcal{K}(\mathcal{H}), \Sigma^{k+n} B \widehat{\otimes} \mathbb{F}_{n,0} \widehat{\otimes}  \mathcal{K}(\mathcal{H}) \rrbracket \\
& =  \varinjlim_n  E_g(A, \Sigma^{k+n} B \widehat{\otimes} \mathbb{F}_{n,0}) \\ 
 & = E_g(A,\Sigma^{k}B) ~ \text{by Proposition~\ref{directlimitofspectra2}}. \\
\end{align*}
\end{proof}

\begin{proposition}
The orthogonal quasi-spectrum $\mathbb{X}(A,B)$ is an $\Omega$-quasi-spectrum. 
\end{proposition}
\begin{proof}
We just need to check that the structure map $ \epsilon \colon E_n \rightarrow \Omega E_{n+1}$ is a weak equivalence. 
That is the map $\pi_k E_n \rightarrow \pi_k \Omega E_{n+1}$ is an isomorphism for all $k$. Now this gives us the map: 
\[ E_g(A, \Sigma^{k} (B \widehat{\otimes} \mathbb{F}_{n,0})) \rightarrow E_g(A, \Sigma^{k+1} (B\widehat{\otimes} \mathbb{F}_{n+1,0})),\]
which is an isomorphism for all $k$ by Theorem~\ref{graded Bott map}.
\end{proof}

\begin{thm}
Let $A,B$ and $C$ be graded $C^\ast$-algebras. Then there is a natural map of orthogonal quasi-spectra
\[\mu_{m,n}\colon \mathbb{X}(A,B) \wedge \mathbb{X}(B,C) \rightarrow \mathbb{X}(A,C),\]
defined by 
\[ (\alpha \wedge \beta)_t \mapsto (\beta_{r(t)} \widehat{\otimes} \text{id}_{\mathbb{F}_{m,0}}) \circ (\text{id}_{\mathcal{S}} \widehat{\otimes}\alpha_t) \circ (\Delta \widehat{\otimes} \text{id}_{A \widehat{\otimes} \mathcal{K}}), \] 
where $\alpha \in \text{Asy}_g(\mathcal{S} \widehat{\otimes} A \widehat{\otimes} \mathcal{K}, B \widehat{\otimes} \mathbb{F}_{m,0} \widehat{\otimes}\mathcal{K})$ and $\beta \in \text{Asy}_g(\mathcal{S} \widehat{\otimes} B \widehat{\otimes}\mathcal{K}, C \widehat{\otimes} \mathbb{F}_{n,0} \widehat{\otimes} \mathcal{K})$. 
In addition the product is associative up to homotopy. 
\end{thm}
\begin{proof}
The product gives a natural $O(m) \times O(n)$-equivariant map: 
\begin{align*}
& \text{Asy}_g (\mathcal{S} \widehat{\otimes} A \widehat{\otimes} \mathcal{K}, B \widehat{\otimes} \mathbb{F}_{m,0} \widehat{\otimes} \mathcal{K}) \wedge \text{Asy}_g (\mathcal{S} \widehat{\otimes} B \widehat{\otimes} \mathcal{K}, C \widehat{\otimes} \mathbb{F}_{n,0} \widehat{\otimes} \mathcal{K}) \\
& \phantom{Magdalini Flari} \longrightarrow \text{Asy}_g (\mathcal{S} \widehat{\otimes} A \widehat{\otimes} \mathcal{K}, C \widehat{\otimes} \mathbb{F}_{m+n,0} \widehat{\otimes} \mathcal{K}), \\
\end{align*}
given by permuting the $m$, $n$ and $m+n$ copies of $\mathbb{F}_{1,0}$.

Now compatibility with the structure maps follows from the naturality of the structure maps and also since we have the following two diagrams: 

\[\xymatrixcolsep{3pc}\xymatrixrowsep{3pc}\xymatrix{
\mathbb{X}_m(A,B) \wedge \mathbb{X}_n(B,C) \ar[d]^-{\epsilon \wedge \text{id}} \ar[r]^-{\mu_{m,n}} &  \mathbb{X}_{m+n}(A,C) \ar[d]^-{\epsilon}\\
\Omega\mathbb{X}_{m+1}(A,B) \wedge \mathbb{X}_n(B,C) \ar[r]^-{\mu_{m+1,n}} & \Omega \mathbb{X}_{m+n+1}(A,C),}\]
and  
\[\xymatrixcolsep{3pc}\xymatrixrowsep{3pc}\xymatrix{
\mathbb{X}_m(A,B) \wedge \mathbb{X}_n(B,C) \ar[d]^-{\text{id} \wedge \epsilon} \ar[r]^-{\mu_{m,n}} &  \mathbb{X}_{m+n}(A,C) \ar[d]^-{\epsilon}\\
\mathbb{X}_{m}(A,B) \wedge \Omega\mathbb{X}_{n+1}(B,C) \ar[r]^-{\mu_{m,n+1}} & \Omega \mathbb{X}_{m+n+1}(A,C),}\]
where $\text{id}$ denote the obvious identities, and the $\epsilon$'s denote the required structure maps. These diagrams commute since,

\begin{eqnarray*}
\lefteqn{\mu_{m+1,n}(\epsilon \wedge \text{id})(\alpha \wedge \beta)_t =  \mu_{m+1,n}(\epsilon(\alpha) \wedge \beta)_t} \\
&=&  (\beta_{r(t)} \widehat{\otimes} \text{id}_{\Sigma  \mathbb{F}_{m+1,0}}) \circ (\text{id}_{\mathcal{S}} \widehat{\otimes} \epsilon(\alpha_t)) \circ (\Delta \widehat{\otimes} \text{id}_{A \widehat{\otimes} \mathcal{K}}) \\
&=&  (\beta_{r(t)} \widehat{\otimes} \text{id}_{\Sigma \mathbb{F}_{m+1,0}}) \circ (\text{id}_{\mathcal{S}} \widehat{\otimes}\left[ (b \widehat{\otimes} \text{id}_{B \widehat{\otimes} \mathbb{F}_{m,0} \widehat{\otimes} \mathcal{K}}) \circ (\text{id}_{\mathcal{S}} \widehat{\otimes} \alpha_t) \circ  (\Delta \widehat{\otimes} \text{id}_{A \widehat{\otimes} \mathcal{K}})\right]) \circ (\Delta \widehat{\otimes} \text{id}_{A \widehat{\otimes} \mathcal{K}}) \\
& =& ( \text{id}_{\Sigma  \mathbb{F}_{1,0}} \widehat{\otimes}\beta_{r(t)}\widehat{\otimes} \text{id}_{\mathbb{F}_{m,0}}) \circ  (b \widehat{\otimes} \text{id}_{\mathcal{S} \widehat{\otimes} B \widehat{\otimes} \mathbb{F}_{m,0} \widehat{\otimes} \mathcal{K}}) \circ (\text{id}_{\mathcal{S} \widehat{\otimes} \mathcal{S}} \widehat{\otimes} \alpha_t)   \circ  (\Delta \widehat{\otimes} \text{id}_{\mathcal{S} \widehat{\otimes} A \widehat{\otimes} \mathcal{K}}) \circ (\Delta \widehat{\otimes} \text{id}_{A \widehat{\otimes} \mathcal{K}})
\\
& =& (b \widehat{\otimes} \text{id}_{C \widehat{\otimes} \mathbb{F}_{m+n,0}\widehat{\otimes} \mathcal{K}}) \circ (\text{id}_{\mathcal{S}} \widehat{\otimes} \beta_{r(t)}\widehat{\otimes} \text{id}_{\mathbb{F}_{m,0}}) \circ (\text{id}_{\mathcal{S} \widehat{\otimes} \mathcal{S}} \widehat{\otimes}\alpha_t)   \\
&& {\phantom{Thanks Mag}} \circ (\Delta \widehat{\otimes} \text{id}_{\mathcal{S} \widehat{\otimes}A \widehat{\otimes} \mathcal{K}})\circ (\Delta \widehat{\otimes} \text{id}_{A \widehat{\otimes} \mathcal{K}})~ \text{by Lemma~\ref{tensorequivalence}}, \\
& =& (b \widehat{\otimes} \text{id}_{C \widehat{\otimes} \mathbb{F}_{m+n,0}\widehat{\otimes} \mathcal{K}}) \circ (\text{id}_{\mathcal{S}} \widehat{\otimes} \\
&& {\phantom{}} \left[(\beta_{r(t)} \widehat{\otimes} \text{id}_{\mathbb{F}_{m,0}}) \circ (\text{id}_{\mathcal{S}} \widehat{\otimes}\alpha_t) \circ (\Delta \widehat{\otimes} \text{id}_{A \widehat{\otimes} \mathcal{K}})\right] ) \circ (\Delta \widehat{\otimes} \text{id}_{A \widehat{\otimes} \mathcal{K}}) \\
& =& \epsilon( (\beta_{r(t)} \widehat{\otimes} \text{id}_{\mathbb{F}_{m,0}}) \circ (\text{id}_{\mathcal{S}} \widehat{\otimes}\alpha_t) \circ (\Delta \widehat{\otimes} \text{id}_{A \widehat{\otimes} \mathcal{K}})) \\
& =& \epsilon(\mu_{m,n}) (\alpha \wedge \beta)_t, \\
\end{eqnarray*}
and
\begin{eqnarray*}
\lefteqn{\mu_{m,n+1}(\text{id} \wedge_g \epsilon)(\alpha \wedge_g \beta) _t = \mu_{m,n+1}(\alpha \wedge \epsilon(\beta))_t} \\
& = &(\epsilon(\beta)_{r(t)} \widehat{\otimes} \text{id}_{\mathbb{F}_{m,0}}) \circ (\text{id}_{\mathcal{S}} \widehat{\otimes}\alpha_t) \circ (\Delta \widehat{\otimes} \text{id}_{A \widehat{\otimes} \mathcal{K}})\\ 
& =& \left[(b \widehat{\otimes} \text{id}_{C \widehat{\otimes} \mathbb{F}_{n,0} \widehat{\otimes} \mathcal{K}}) \circ (\text{id}_{\mathcal{S}} \widehat{\otimes} \beta_{r(t)}  \circ (\Delta \widehat{\otimes} \text{id}_{B \widehat{\otimes} \mathcal{K}})\right]\widehat{\otimes} \text{id}_{\mathbb{F}_{m,0}}) \circ (\text{id}_{\mathcal{S}} \widehat{\otimes}\alpha_t) \circ (\Delta \widehat{\otimes} \text{id}_{A \widehat{\otimes} \mathcal{K}})\\ 
& = & (b \widehat{\otimes} \text{id}_{C \widehat{\otimes} \mathbb{F}_{m+n,0}\widehat{\otimes} \mathcal{K}}) \circ (\text{id}_{\mathcal{S}} \widehat{\otimes} \beta_{r(t)}\widehat{\otimes} \text{id}_{\mathbb{F}_{m,0}}) \circ  (\Delta \widehat{\otimes} \text{id}_{B \widehat{\otimes} \mathbb{F}_{m,0} \widehat{\otimes} \mathcal{K}}) \circ (\text{id}_{\mathcal{S}} \widehat{\otimes}\alpha_t) \circ (\Delta \widehat{\otimes} \text{id}_{A \widehat{\otimes} \mathcal{K}})\\ 
& =& (b \widehat{\otimes} \text{id}_{C \widehat{\otimes} \mathbb{F}_{m+n,0}\widehat{\otimes} \mathcal{K}}) \circ (\text{id}_{\mathcal{S}} \widehat{\otimes} \beta_{r(t)}\widehat{\otimes} \text{id}_{\mathbb{F}_{m,0}}) \circ (\text{id}_{\mathcal{S} \widehat{\otimes} \mathcal{S}} \widehat{\otimes}\alpha_t)  \\
& &{\phantom{Thanks Mag}}\circ (\Delta \widehat{\otimes} \text{id}_{\mathcal{S} \widehat{\otimes}A \widehat{\otimes} \mathcal{K}}) \circ (\Delta \widehat{\otimes} \text{id}_{A \widehat{\otimes} \mathcal{K}})~ \text{by Lemma~\ref{tensorequivalence}},\\
& =& (b \widehat{\otimes} \text{id}_{C \widehat{\otimes} \mathbb{F}_{m+n,0}\widehat{\otimes} \mathcal{K}}) \circ (\text{id}_{\mathcal{S}} \widehat{\otimes}  \left[(\beta_{r(t)} \widehat{\otimes} \text{id}_{\mathbb{F}_{m,0}}) \circ (\text{id}_{\mathcal{S}} \widehat{\otimes}\alpha_t) \circ (\Delta \widehat{\otimes} \text{id}_{A \widehat{\otimes} \mathcal{K}})\right] ) \circ (\Delta \widehat{\otimes} \text{id}_{A \widehat{\otimes} \mathcal{K}})\\
& =& \epsilon( (\beta_{r(t)} \widehat{\otimes} \text{id}_{\mathbb{F}_{m,0}}) \circ (\text{id}_{\mathcal{S}} \widehat{\otimes}\alpha_t) \circ (\Delta \widehat{\otimes} \text{id}_{A \widehat{\otimes} \mathcal{K}})) \\
& =& \epsilon(\mu_{m,n}) (\alpha \wedge \beta)_t ,\\
\end{eqnarray*}
for all $\alpha \wedge \beta \in \mathbb{E}_m(A,B) \wedge \mathbb{E}_n(B,C)$.


Now we check that our product is associative up to homotopy. 

Let $\alpha \in \text{Asy}_g (\mathcal{S} \widehat{\otimes} A \widehat{\otimes} \mathcal{K},  B \widehat{\otimes} \mathbb{F}_{m,0} \widehat{\otimes} \mathcal{K})$ and $\beta \in \text{Asy}_g (\mathcal{S} \widehat{\otimes}  B \widehat{\otimes} \mathcal{K}, C \widehat{\otimes} \mathbb{F}_{n,0} \widehat{\otimes} \mathcal{K})$ and $\gamma \in \text{Asy}_g (\mathcal{S} \widehat{\otimes}  C \widehat{\otimes} \mathcal{K},  D \widehat{\otimes} \mathbb{F}_{p,0} \widehat{\otimes} \mathcal{K})$. Then take the homotopy classes of these elements and we obtain $E$-theory groups, and we know that the $E$-theory product is associative. 

\end{proof}

\section{Connecting graded $K$ and $E$-theory spectra}\label{$K$-theory}
This section connects together graded $K$-theory and $E$-theory spectra. In particular we form a smash product in terms of these two spectra and consequently combine $K$-theory and $K$-homology in a smash product. 
Here we obtain a connection between $K$-theory, $K$-homology and $E$-theory spectra. 
Let $\text{Hom}(A,B)$ denote the set of graded $\ast$-homomorphisms from $A$ to $B$. Then recall that $$K_n(A) = [\mathcal{S}, \Sigma^n A \widehat{\otimes} \mathcal{K}(\mathcal{H})] \cong E^n_g(\mathbb{F}, A),$$ 
and the $K$-homology is given by
$$\mathbb{K}_{\text{hom}}(A) = E_n(A, \mathbb{F}).$$

\subsection{A topology on graded $\ast$-homomorphisms}

Let $\text{Hom}_g(A,B)$ denote the set of graded $\ast$-homomorphisms from $A$ to $B$.
We equip $\text{Hom}_g(A,B)$ with the compact open topology as detailed below.

\begin{definition}
A \emph{basis} for a topology on a set $A$ is a collection of subsets $\mathcal{A}$ of $A$ such that $A$ is a union of sets from $\mathcal{A}$ and such that if $A_1, A_2$ are in $\mathcal{A}$ then their intersection is a union of sets from $\mathcal{A}$.  

A \emph{subbasis} is a collection of subsets $\mathcal{B}$ of $A$ where the set $\mathcal{A}$ of all finite intersections of sets in $\mathcal{B}$ is a basis. 
\end{definition}
\begin{definition}
Let $A$ and $B$ be graded $C^\ast$-algebras. The \emph{compact open topology} on the set of graded $\ast$-homomorphisms from $A$ to $B$, $\text{Hom}_g(A, B)$, is generated by subsets of the following form, 
\[B(K,U) = \{f  \in \text{Hom}_g(A, B)~| ~f(K) \subset U\},\]
where $K$ is compact in $A$ and $U$ is open in $B$. 
Here \emph{generated} means that the sets defined form a subbasis for the open sets when we think of a topology. This then generates a basis for the topology. 
\end{definition}

For simplicity of notation let $\text{Hom}_g( A, B)$ denote the space of graded $\ast$-homomorphisms from $A$ to $B$ equipped with the compact open topology. 
Denote the loop space of this space by $\Omega \text{Hom}_g( A, B)$. Note that the basepoints for both of these spaces is just the zero $\ast$-homomorphism, which will we denote by $0$.

The compact open topology is the choice for our topology since it gives us the correct path components for our loop space and it also allows us to have continuity of particular maps as we will see soon. 

Now let us consider the generators of the compact open topology on the spaces $\Omega \text{Hom}_g( A, B )$ and $\text{Hom}_g (A, \Sigma B)$. 
Now we have a basis for $\text{Hom}_g( A, B)$, so we just extend this for  $\Omega \text{Hom}_g ( A, B )$, and it is not to hard to see that a basis for the loop space is the set generated by $B(K', V)$ such that $K' \subseteq [0,1]$ compact and $V \subseteq \text{Hom}_g(A,B)$ open.

Combining these, we obtain the following definition. 

\begin{definition}
The \emph{compact open topology} on $\Omega \text{Hom}_g( A, B )$ is generated by sets of the form $B(K', B(K,U))$, where $K \subseteq A$ compact, $K' \subseteq [0,1]$ compact and $U \subseteq B$ open. 
The \emph{compact open topology} on $\text{Hom}_g (A, \Sigma B)$ is generated by sets of the form $B(K, B(K',U))$ where $K \subseteq A$ compact, $K' \subseteq [0,1]$ compact  and $U \subseteq B$ open. 
\end{definition}

Before we check we have the continuity of maps in the following proof, it is worth noting that it is sufficient to check that a map is continuous under a topology by considering a subbasis. That is, to check a map of topological spaces is continuous we just need to check that continuity holds at the level of generating sets for a basis of a topology. For details, see~\cite{Sut75}, Application 3.2.5.

\begin{proposition}
The spaces $\Omega \text{Hom}_g( A, B )$ and $\text{Hom}_g (A, \Sigma B)$ are homeomorphic. 
\end{proposition}
\begin{proof}
We consider ungraded $\ast$-homomorphisms since the grading property is immediate. 
Define $f: \Omega \text{Hom}( A, B ) \rightarrow\text{Hom} (A, \Sigma B)$ as follows. Let $\mu \in \Omega \text{Hom}( A, B )$ based at $0$, then  define 
\[ f(\mu) (a)(s) =  \mu(s)(a),\]
for all $a \in A$ and $s \in [0,1]$. 

Now define $g:  \text{Hom} (A, \Sigma B) \rightarrow  \Omega \text{Hom}( A, B )$ as follows. Let $\tau \in  \text{Hom} (A, \Sigma B)$, define
\[g( \tau )(s)(a) = \tau(a)(s).\] 

Both $f$ and $g$ are well defined since $\mu$ and $\tau$ are $\ast$-homomorphisms. 

We need to show that $f \circ g = id$, and $g \circ f = id$ where $id$ stands for the natural identities. 

Let $\varphi \in \text{Hom} (A, \Sigma B)$, then for all $a \in A, s \in [0,1]$. 
\[fg(\varphi) (a)(s) = g(\varphi)(s)(a) = \varphi(a)(s).\] 

Similarly, let $\psi \in \Omega \text{Hom}( A, B )$, then for all $a \in A, s \in [0,1]$, 
\[gf(\psi)(s)(a) = f(\psi)(a)(s) = \psi(s)(a).\]
Then $f \circ g = id$ and $g \circ f = id$ as required.

Now we check that $f$ and $g$ are continuous. By the above discussion, it suffices to check that: 
\[f^{-1} [B(K', B(K,U))] = B(K, B(K',U))\]
and
\[g^{-1} [B(K, B(K',U))] = B(K', B(K,U)),\]
for all $K \subseteq A$ compact, $K' \subseteq [0,1]$ compact and $U \subseteq B$ open. 

Let $y \in  B(K, B(K',U))$, then $f^{-1}(y) = \{ x ~ | ~ f(x) = y\}$. Now let $x \in f^{-1}(y)$, then we know 
\[f(x)(s)(a) = x(a)(s) = y(s)(a),\]
so $x$ must be contained in $B(K', B(K,U))$, and similarly we can check the converse, so $f$ is continuous. 

Similarly we can prove that $g$ is continuous. 
\end{proof}

\subsection{$K$-theory spectra}
Now we can define the $K$-theory spectrum. 

\begin{definition} \label{K-theory spectrum 1}
Let $\mathcal{K} = \mathcal{K}(\mathcal{H})$.
Define $\mathbb{K}(A)$ to be the sequence of based topological spaces $$K_n= \text{Hom}_g (\mathcal{S} \widehat{\otimes} \mathbb{F} \widehat{\otimes} \mathcal{K}, A \widehat{\otimes} \mathbb{F}_{n,0} \widehat{\otimes} \mathcal{K})$$ where $m \geq 0$. Define maps $\epsilon_n \colon K_n \rightarrow \Omega K_{n+1}$:
  \[\xymatrix{\text{Hom}_g( \mathcal{S} \widehat{\otimes} \mathbb{F} \widehat{\otimes} \mathcal{K}, A \widehat{\otimes} \mathbb{F}_{n,0} \widehat{\otimes} \mathcal{K})  \ar[r]  & \Omega \text{Hom}_g(\mathcal{S} \widehat{\otimes} \mathbb{F} \widehat{\otimes}\mathcal{K},  A \widehat{\otimes} \mathbb{F}_{n+1,0} \widehat{\otimes} \mathcal{K}) \ar@{=}[d]^{\cong} \\ &  \text{Hom}_g(\mathcal{S} \widehat{\otimes} A \widehat{\otimes} \mathcal{K}, \Sigma ( A \widehat{\otimes} \mathbb{F}_{n+1,0}) \widehat{\otimes} \mathcal{K})}   \]
by: 
\[\epsilon (x_t) = (b \widehat{\otimes} \text{id}_{A \widehat{\otimes} \mathbb{F}_{n,0} \widehat{\otimes} \mathcal{K}}) \circ (\text{id}_{\mathcal{S}} \widehat{\otimes} x_t)  \circ (\Delta \widehat{\otimes} \text{id}_{\mathbb{F} \widehat{\otimes}\mathcal{K}}),\]
for all $x_t \in\text{Asy}_g (\mathcal{S} \widehat{\otimes} A \widehat{\otimes} \mathcal{K}, B \widehat{\otimes} \mathbb{F}_{n,0} \widehat{\otimes} \mathcal{K})$ and the Bott map $b \in \text{Hom}_g(\mathcal{S},\Sigma  \mathbb{F}_{1,0})$. 

\end{definition}

We now give an alternative definition for the spectrum of graded $K$-theory in terms of asymptotic morphisms. 

\begin{definition} \label{K-theory spectrum 2}
Let $\mathcal{K} = \mathcal{K}(\mathcal{H})$.
Define $\mathbb{K}'(A)$ to be the orthogonal quasi-spectrum with the sequence of based quasi-topological spaces $$K'_n = \text{Asy}_g (\mathcal{S} \widehat{\otimes} \mathbb{F} \widehat{\otimes} \mathcal{K}, A \widehat{\otimes} \mathbb{F}_{n,0} \widehat{\otimes} \mathcal{K})$$ where $n \geq 0$. The structure maps $\epsilon \colon K'_n \rightarrow \Omega K'_{n+1}$: 
  \[\xymatrix{\text{Asy}_g( \mathcal{S} \widehat{\otimes} \mathbb{F} \widehat{\otimes} \mathcal{K}, A \widehat{\otimes} \mathbb{F}_{n,0} \widehat{\otimes} \mathcal{K})  \ar[r]  & \Omega \text{Asy}_g(\mathcal{S} \widehat{\otimes} \mathbb{F} \widehat{\otimes}\mathcal{K},  A \widehat{\otimes} \mathbb{F}_{n+1,0} \widehat{\otimes} \mathcal{K}) \ar@{=}[d]^{\cong} \\ &  \text{Asy}_g(\mathcal{S} \widehat{\otimes} \mathbb{F} \widehat{\otimes} \mathcal{K}, \Sigma ( A \widehat{\otimes} \mathbb{F}_{n+1,0}) \widehat{\otimes} \mathcal{K})}   \]
are defined by: 
\[\epsilon (x_t) = (b \widehat{\otimes} \text{id}_{A \widehat{\otimes} \mathbb{F}_{n,0} \widehat{\otimes} \mathcal{K}}) \circ (\text{id}_{\mathcal{S}} \widehat{\otimes} x_t)  \circ (\Delta \widehat{\otimes} \text{id}_{\mathbb{F} \widehat{\otimes}\mathcal{K}}),\]
for all $x_t \in\text{Asy}_g (\mathcal{S} \widehat{\otimes} A \widehat{\otimes} \mathcal{K}, B \widehat{\otimes} \mathbb{F}_{n,0} \widehat{\otimes} \mathcal{K})$ and the Bott map $b \in \text{Hom}_g(\mathcal{S},\Sigma\mathbb{F}_{1,0})$. 
\end{definition}

We now notice that Definition~\ref{K-theory spectrum 1} and Definition~\ref{K-theory spectrum 2} are orthogonal and orthogonal quasi-spectra for the same reason that~\ref{E-theory spectrum} forms one and consequently the following result comes from the stable homotopy groups coming from these spectra. 

\begin{proposition}
The map of spectrum $f \colon \mathbb{K}(A) \rightarrow \mathbb{K}'(A)$, defined by $f(\varphi) = \varphi$ for all $\varphi \in \mathbb{K}(A)$ is a weak equivalence. 
\end{proposition}
\begin{proof}
Consider the map $f' \colon \text{Hom}_g(\mathcal{S}, A) \rightarrow \text{Asy}_g(\mathcal{S}, A)$ then this induces the map $f_{\ast}' \colon [\mathcal{S}, A] \rightarrow \llbracket \mathcal{S}, A \rrbracket$. 
Now the map, 
\[\text{Hom}_g(\mathcal{S}, A \widehat{\otimes} \mathbb{F}_{n+1,0}) \rightarrow \text{Asy}_g(\mathcal{S}, A \widehat{\otimes} \mathbb{F}_{n+1,0}),\] induces an isomorphism at the level of $\pi_0$, and therefore the map 
\[\text{Hom}_g(\mathcal{S},A \widehat{\otimes} \mathbb{F}_{1,0}) \rightarrow \text{Asy}_g(\mathcal{S},A \widehat{\otimes}\mathbb{F}_{1,0}),\]
induces an isomorphism at the level of $\pi_n$. 
Therefore we have a weak equivalence. 
Then we can also consider the map $f$ above and the same applies, since we obtain this map by tensoring with the suspension and the complex numbers. 
\end{proof}

\begin{corollary}\label{inversespectra}
The map of spectrum $f \colon \mathbb{K}(A) \rightarrow \mathbb{K}'(A)$ has a natural inverse $g \colon \mathbb{K}' \rightarrow \mathbb{K}$ at the level of stable homotopy groups.
\end{corollary}
\qed

\begin{thm}
Let $A$ and $B$ be $C^\ast$-algebras. Then there is a natural map of orthogonal quasi-spectra
\[\nu'_{m,n}\colon \mathbb{K}(A) \wedge \mathbb{E}(A,B) \rightarrow \mathbb{K}'(B),\]
defined by 
\[ (\alpha \wedge \beta_t)_t \mapsto (\beta_{t} \widehat{\otimes} \text{id}_{\mathbb{F}_{m,0}}) \circ (\text{id}_{\mathcal{S}} \widehat{\otimes}\alpha) \circ (\Delta \widehat{\otimes} \text{id}_{\mathbb{F} \widehat{\otimes} \mathcal{K}(\mathcal{H})}), \] 
where $\alpha \in \text{Hom}_g (\mathcal{S} \widehat{\otimes} A \widehat{\otimes} \mathcal{K}(\mathcal{H}), B \widehat{\otimes} \mathbb{F}_{m,0} \widehat{\otimes}\mathcal{K}(\mathcal{H}))$ and $\beta \in \text{Asy}_g (\mathcal{S} \widehat{\otimes} B \widehat{\otimes}\mathcal{K}(\mathcal{H}), C \widehat{\otimes} \mathbb{F}_{n,0} \widehat{\otimes} \mathcal{K}(\mathcal{H}))$. 
\end{thm}
\begin{proof}
Since the composition of a $\ast$-homomorphism and an asymptotic morphism is an asymptotic morphism it is clear that $\alpha \wedge \beta$ is an asymptotic morphism and lies in the required spectra.  
\end{proof}

By the above theorem and Corollary~\ref{inversespectra}, we obtain
\begin{corollary}
There is a natural map of spectra
\[\nu_{m,n}\colon \mathbb{K}(A) \wedge \mathbb{E}(A,B) \rightarrow \mathbb{K}(B),\]
with the above criteria.
\end{corollary}
\qed

Now we finalise this section by combining the graded $K$-theory spectrum and $K$-homology spectrum noting that $\mathbb{K}_{\text{hom}} (A) = \mathbb{E} (A, \mathbb{F})$.

\begin{thm}
There is a canonical map 
\[S\colon \mathbb{K}(A \widehat{\otimes} B) \wedge \mathbb{K}_{\text{hom}}(A) \rightarrow \mathbb{K}'(B)\]
of orthogonal quasi-spectra. The map $S$ is natural in the variable $B$ in the obvious sense and natural in the variable $A$, in the sense that if we have a $\ast$-homomorphism $f \colon A \rightarrow A'$ then we have the following commutative diagram 
\[\xymatrixcolsep{2pc}\xymatrixrowsep{2pc}\xymatrix{
\mathbb{K}(A \widehat{\otimes} B) \wedge \mathbb{K}_\text{hom}(A)  \ar[r]^-{S} &  \mathbb{K}'(B) \ar@{=}[d]\\
\mathbb{K}(A \widehat{\otimes} B) \wedge \mathbb{K}_\text{hom}(A') \ar[u]^-{\text{id} \wedge f^{\ast}} \ar[d]_{f_{\ast} \wedge \text{id}}  & \mathbb{K}'(B)\\
\mathbb{K}(A' \widehat{\otimes}B) \wedge\mathbb{K}_\text{hom}(A') \ar[r]^-{S} & \mathbb{K}'(B) \ar@{=}[u],}\]
where $f_{\ast}$  and $f^\ast$ are defined by: 
\[f_{\ast}(\alpha)(\lambda) = (f \widehat{\otimes} \text{id}_{B \widehat{\otimes} \mathbb{F}_{m,0} \widehat{\otimes} \mathcal{K}(\mathcal{H})})(\alpha(\lambda)),\]
and 
\[f^{\ast}(\beta_t)(a) = \beta_t(f \widehat{\otimes} \text{id}_{\mathcal{S} \widehat{\otimes} \mathcal{K}(\mathcal{H})}) (a), \]
with $\alpha \in \mathbb{K}(A \widehat{\otimes} B)$, $\beta \in \mathbb{K}_\text{hom}(A')$, $a \in \mathcal{S} \widehat{\otimes} A \widehat{\otimes} \mathcal{K}$ and $\lambda \in \mathcal{S} \widehat{\otimes} \mathbb{F} \widehat{\otimes} \mathcal{K}$.
\end{thm}
\begin{proof}
Writing $\mathbb{K}_\text{hom}(A) = \mathbb{E}(A,\mathbb{F})$, we can extend the definition of $S$ to a composition of maps, in order to obtain the following diagram:
\[\xymatrixcolsep{3pc}\xymatrixrowsep{2pc}\xymatrix{
\mathbb{K}(A \widehat{\otimes} B) \wedge \mathbb{E}(A,\mathbb{F})  \ar[r]^-{id \wedge \widehat{\otimes}_B} & \mathbb{K}(A \widehat{\otimes} B) \wedge \mathbb{E}(A \widehat{\otimes} B,B) \ar[r]^-{\nu'_{m,n}}& \mathbb{K}'(B) \ar@{=}[d]\\
\mathbb{K}(A \widehat{\otimes} B) \wedge\mathbb{E}(A',\mathbb{F}) \ar[u]^-{\text{id} \wedge f^{\ast}} \ar[d]_{f_{\ast} \wedge \text{id}} & & \mathbb{K}'(B)\\
\mathbb{K}(A' \widehat{\otimes} B) \wedge\mathbb{E}(A',\mathbb{F}) \ar[r]^-{\text{id} \wedge \widehat{\otimes}_B} &  \mathbb{K}(A' \widehat{\otimes} B) \wedge \mathbb{E}(A' \otimes B,B) \ar[r]^-{\nu'_{m,n}} & \mathbb{K}'(B) \ar@{=}[u],}\]
\[\mathbb{K}(A \widehat{\otimes} B) \wedge \mathbb{E}(A \widehat{\otimes} B, B) \rightarrow \mathbb{K}'(B).\]
Then we have

\begin{eqnarray*}
\lefteqn{\nu'_{m,n}(\text{id} \wedge \widehat{\otimes}_B)(\text{id} \wedge f^{\ast})(\alpha \wedge \beta_t)   =  \nu'_{m,n}(\text{id} \wedge \widehat{\otimes}_B)(\alpha \wedge f^{\ast}(\beta_t))} \\
&\phantom{lets s} = &  \nu'_{m,n}(\alpha \wedge f^{\ast}(\beta_t) \widehat{\otimes} \text{id}_B)  \\
&\phantom{lets s} =& ( f^{\ast}(\beta_t) \widehat{\otimes} \text{id}_{B \widehat{\otimes} \mathbb{F}_{m,0}}) \circ (\text{id}_{\mathcal{S}} \widehat{\otimes}\alpha) \circ (\Delta \widehat{\otimes} \text{id}_{\mathbb{F} \widehat{\otimes} \mathcal{K}(\mathcal{H})}) \\
&\phantom{lets s} =& ([\beta_t \circ (f \widehat{\otimes} \text{id}_{\mathcal{S} \widehat{\otimes} \mathcal{K}(\mathcal{H})})] \widehat{\otimes} \text{id}_{B \widehat{\otimes} \mathbb{F}_{m,0}}) \circ (\text{id}_{\mathcal{S}} \widehat{\otimes}\alpha) \circ (\Delta \widehat{\otimes} \text{id}_{\mathbb{F} \widehat{\otimes} \mathcal{K}(\mathcal{H})}) \\
&\phantom{lets s} =& (\beta_t \widehat{\otimes} \text{id}_{B \widehat{\otimes} \mathbb{F}_{m,0}}) \circ  (f \widehat{\otimes} \text{id}_{\mathcal{S} \widehat{\otimes} {B \widehat{\otimes} \mathbb{F}_{m,0}}  \widehat{\otimes} \mathcal{K}(\mathcal{H})}) \circ (\text{id}_{\mathcal{S}} \widehat{\otimes}\alpha) \circ (\Delta \widehat{\otimes} \text{id}_{\mathbb{F} \widehat{\otimes} \mathcal{K}(\mathcal{H})}) \\
&\phantom{lets s} = & (\beta_t \widehat{\otimes} \text{id}_{B \widehat{\otimes} \mathbb{F}_{m,0}}) \circ (\text{id}_{\mathcal{S}} \widehat{\otimes} [(f \widehat{\otimes} \text{id}_{{B \widehat{\otimes} \mathbb{F}_{m,0}}  \widehat{\otimes} \mathcal{K}(\mathcal{H})}) \circ \alpha] \circ (\Delta \widehat{\otimes} \text{id}_{\mathbb{F} \widehat{\otimes} \mathcal{K}(\mathcal{H})}) \\
&\phantom{lets s} =& (\beta_t \widehat{\otimes} \text{id}_{B \widehat{\otimes} \mathbb{F}_{m,0}}) \circ (\text{id}_{\mathcal{S}} \widehat{\otimes} f_{\ast}(\alpha)) \circ (\Delta \widehat{\otimes} \text{id}_{\mathbb{F} \widehat{\otimes} \mathcal{K}(\mathcal{H})}) \\ 
&\phantom{lets s} =& v'_{m,n}(f_{\ast}(\alpha) \wedge (\beta_t \widehat{\otimes} \text{id}_B)) \\
&\phantom{lets s} =& v'_{m,n}(\text{id} \wedge \widehat{\otimes}_B) (f_{\ast}(\alpha) \wedge \beta_t) = v'_{m,n}(\text{id} \wedge \widehat{\otimes}_B) ( f_{\ast} \wedge \text{id})(\alpha \wedge \beta_t). \\
\end{eqnarray*}
\end{proof}

\begin{corollary}
There is a canonical map 
\[S \colon \mathbb{K}(A \otimes B) \wedge \mathbb{K}_{\text{hom}}(A) \rightarrow \mathbb{K}(B)\]
of orthogonal quasi-spectra. The map $S$ is natural in the variable $B$ in the obvious sense and natural in the variable $A$, in the sense that if we have a $\ast$-homomorphism $f \colon A \rightarrow A'$ then we have the following commutative diagram 
\[\xymatrixcolsep{2pc}\xymatrixrowsep{2pc}\xymatrix{
\mathbb{K}(A \otimes B) \wedge \mathbb{K}_\text{hom}(A)  \ar[r]^-{S} &  \mathbb{K}(B) \ar@{=}[d]\\
\mathbb{K}(A \otimes B) \wedge \mathbb{K}_\text{hom}(A') \ar[u]^-{\text{id} \wedge f^{\ast}} \ar[d]_{f_{\ast} \wedge \text{id}}  & \mathbb{K}(B)\\
\mathbb{K}(A' \otimes B) \wedge\mathbb{K}_\text{hom}(A') \ar[r]^-{S} & \mathbb{K}(B). \ar@{=}[u],}\]

\end{corollary}
\qed

\end{document}